\documentclass[11pt]{amsart}

\usepackage[T1]{fontenc}
\usepackage{tgschola}

\usepackage{amssymb, amscd, amsmath, amsthm, epsf, epsfig, latexsym,color} 
\usepackage{hyperref}

\usepackage{comment}

\usepackage{pinlabel}

 \usepackage{tikz-cd}

 \newcommand{\Hom}{\operatorname{Hom}}

\newcommand{\TryPackage}[3]{\IfFileExists{#1.sty}{\usepackage{#1}#2}{#3}
}
\TryPackage{mathrsfs}{\renewcommand{\mathcal}{\mathscr}}{%
            \TryPackage{eucal}{}{}}

\newcommand{\al}{\alpha}
\newcommand{\be}{\beta}

\newcommand{\ep}{\epsilon}

\newcommand{\bbi}{{{\bf i}}}
\newcommand{\bbj}{{{\bf j}}}
\newcommand{\bbk}{{{\bf k}}}

\newcommand{\ZZ}{{\mathbb Z}}
\newcommand{\RR}{{\mathbb R}}
\newcommand{\CC}{{\mathbb C}}
\newcommand{\HH}{{\mathbb H}}

\newcommand{\TT}{{\mathbb T}}
\newcommand{\BB}{{\mathbb B}}
\newcommand{\SSS}{{\mathbb S}}
\newcommand{\ab}{{\hbox{\scriptsize \sl ab}}}
\newcommand{\conju}{{\hbox{\scriptsize \sl conj}}}
\newcommand{\cen}{{\hbox{\scriptsize \sl cen}}}
\newcommand{\anc}{{\hbox{\scriptsize \sl anc}}}
\newcommand{\Ima}{\operatorname{Im}}

\newcommand{\Stab}{\operatorname{Stab}}
\newcommand{\Real}{\operatorname{Re}}

\def\del{\partial}

\def\del{\partial}

\theoremstyle{definition}
\newtheorem{df}{Definition}[section]
\theoremstyle{plain}
\newtheorem{theorem}[df]{Theorem}

\newtheorem{corollary}[df]{Corollary}
\newtheorem{lemma}[df]{Lemma}
\newtheorem{proposition}[df]{Proposition}

 \newtheorem{theoremABC}{Theorem}

\author{Christopher M. Herald}
\address{Department of Mathematics and Statistics, University of Nevada,  Reno, NV 89557} 
\email{herald@unr.edu}
 \author{Paul Kirk}
\address{Department of Mathematics, Indiana University,
 Bloomington, IN 47405} 
\email{pkirk@iu.edu}
 \thanks{CH was supported by  a Simons Collaboration Grant for Mathematicians.}


\begin{document}

\title[$SU(2)$ character variety of the genus two surface] {A  differential topology proof that the $SU(2)$ character variety of the genus two surface is homeomorphic to $\CC P^3$}

\begin{abstract}   
We provide  a proof that  the $SU(2)$ character variety of a genus  two surface, $\chi(F_2)$,  is  a  closed compact manifold,  and    a  proof
of the Narasimhan-Ramanan theorem that $\chi(F_2)$  is 
 homeomorphic to  $\CC P^3$. This is done entirely in the language of $SU(2)$ representations,  differential topology and elementary algebraic topology.  
It avoids the Narasimhan-Seshadri correspondence,   clarifying the nature of Lagrangian immersions into $\chi(F_2)$ induced by 3-manifolds with genus  two boundary. We  give examples   of such Lagrangian immersions and describe a   correspondence from multicurves in the pillowcase to Lagrangian immersions in $\chi(F_2)$, induced by a 2-stranded tangle in a  punctured genus 2 handlebody.  We give an example of a non-transverse pair  of smooth Lagrangians in $\chi(F_2)$ induced by a genus 2 Heegaard splitting of $(S^3,W)$ for the ``linked eyeglasses" web $W$, which are made transverse, and hence the corresponding Chern-Simons function Morse,  using  Goldman flows/holonomy perturbations along embedded curves in the Heegaard surface.
\end{abstract}

 \maketitle
%%%%%%%%%%%%%%%%%%% 

 \section{Introduction}
 
The {\em SU(2) character variety} of the closed oriented genus  two surface $F_2$ is the  space (a compact semi-algebraic set over $\RR$)
$$\chi(F_2)=\Hom(\pi_1(F_2),SU(2))/_{\rm conjugation}.$$
 It is partitioned as
 $$\chi(F_2)=\chi(F_2)^*\sqcup \chi^\ab(F_2)$$
where $\chi^*(F_2)\subset \chi(F_2)$ denotes the subset of conjugacy classes of homomorphisms with non-abelian image,   and $\chi(F_2)^\ab\subset \chi(F_2)$ those with abelian image.
  It is well known  that $\chi^*(F_2)$ is  a smooth manifold of dimension $6$ 
 (see Proposition \ref{dimension}).

\begin{theoremABC}\label{thmA}    Let $C$ be an embedded curve which separates $F_2$ into two punctured tori, and let $\kappa\colon\chi(F_2)\to [-1,1]$ send a conjugacy class $[\rho]$ to $\tfrac1 2 {\rm Trace}(\rho(C)).$
Then the decomposition
$$\chi(F_2)=\{\kappa\leq 0\}\cup_{\{\kappa=0\}}\{\kappa\ge0\}$$   satisfies 
\begin{itemize}
\item $\{\kappa\leq 0\}$  is a closed tubular neighborhood of $\{ \kappa = -1\}
$ in  $\chi^*(F_2)$,  diffeomorphic to $\RR P^3\times D^3$. 
\item $\{\kappa\ge0\}$ is homeomorphic to the total space of a smooth, oriented,  Euler class $(2,2)$  $D^2$ bundle $E\to S^2\times S^2$, by a homeomorphism which is a diffeomorphism  on $\{ 0\leq \kappa\leq 0.5\}\subset \chi^*(F_2)$.  
\end{itemize}
The abelian  locus $\chi^\ab(F_2)$  is  contained in $\{\kappa=1\}$. 
\end{theoremABC}

\begin{theoremABC}\label{thmB}
The $SU(2)$ character variety  $\chi(F_2)$ of a genus  two surface is  homeomorphic to $\CC P^3$.
\end{theoremABC}
 
Our proof of Theorem \ref{thmB},  a well-known 1969  theorem of Narasimhan and Ramanan \cite[Theorem 2]{NR}, combines Theorem \ref{thmA} with a simple application  of  Wall's  classification of  smooth simply connected  spin
6-manifolds \cite{Wall-classificationV}.
We show in Section \ref{POT8}  that   $\{\kappa=-1\}$ is diffeomorphic to $\RR P^3$; our homeomorphism with $\CC P^3$ sends $\{\kappa=-1\}$ to the real points  $\RR P^3$.

\medskip 

The proof of Theorem \ref{thmA} is carried out in two steps. First, we give an equivariant differential topology proof that $\chi(F_2)$ is a topological manifold by showing that points of each orbit type have neighborhoods homeomorphic to $\RR^6$ (Theorems \ref{th4.1K6} and \ref{thm2.2}).    Second, we show $\kappa$ has only two critical values and  identify the level sets and the sup- and sub-level sets of the character $\kappa$.
Our main technical challenge in the second step is the proof of Theorem \ref{thm8}, which establishes the existence of a homeomorphism of the sup-level set $\{\kappa\ge 0\}$  with  a neighborhood of the smooth quadratic complex surface $Q$ in $\CC P^3$.

\subsection{Discussion} Narasimhan and Ramanan's   proof   of Theorem \ref{thmB} passes through the Narasimhan-Seshadri correspondence  \cite{NS},
which asserts that the $SU(2)$ character variety $\chi(F)$ of a Riemann surface $F$ is homeomorphic to the moduli space of trivial determinant, semi-stable, rank two holomorphic bundles over $F$.   This latter moduli space is identified, in the case of any genus two Riemann surface,  with  the smooth  complex manifold $\CC P^3$  in \cite{NR}.   This description of $\chi(F_2)$ is incomplete from the perspective of Floer theory and low-dimensional topology, which aims to  extract information  from  
  Lagrangian immersions  into $\chi^*(F_2)$ induced by  compact 3-manifolds with boundary $F_2$.
Understanding what these Lagrangian immersions look like ``in $\CC P^3$'' is tricky since the Narasimhan-Seshadri correspondence  has  no (direct)  3-dimensional counterpart.

 \medskip

 Our arguments rely instead on  an analysis of $SU(2)$ representations, avoiding  the Narasimhan-Seshadri correspondence and the Narasimhan-Ramanan theorem.   This approach facilitates  identifying Lagrangian immersions $\chi(Y)\to \chi(F_2)$  induced by 3-manifolds $Y$ with genus two boundary.  
 
  We identify such immersions for various choices of 3-manifolds in the last third of this article (Section \ref{lagrange}),  taking an approach consistent with the Atiyah-Floer conjecture \cite{Atiyah1} and holonomy perturbations of Chern-Simons functions.   
 In Corollary \ref{correspond} we describe the Lagrangian correspondence
 from multicurves in the 4-punctured 2-sphere   to immersed Lagrangian 3-manifolds in $\chi(F_2)$ 
   induced by the $SU(2)$ character variety of the tangle in a punctured genus two handlebody illustrated in Figure \ref{fig00fig}.  
In Section \ref{HOE} we illustrate  the explicit relationship between holonomy perturbations and Goldman flows (established in \cite{HK2}) 
 to    make a    pair of Lagrangians in $\chi(F_2)$ associated to a Heeegard splitting transverse.   
  A good introduction to the mathematical ideas which motivate the construction of these examples can be found in Atiyah's    article {\em New Invariants of 3- and 4-Dimensional Manifolds} \cite{Atiyah1}.

   The homeomorphism   of $\chi(F_2)$ with $ \CC P^3$ we provide is explicit enough to  
parameterize  embedded submanifolds of $\chi(F_2)$ representing classes in $H_*(\chi(F_2))\cong H_*(\CC P^3)$. For example, Let $Q'\subset \chi(F_2)$ denote 
the  conjugacy classes  $\rho\colon \pi_1(F_2)\to SU(2)$     taking  the first two standard surface generators into the maximal torus $\{e^{\theta\bbi}\}$ and the third and fourth into $\{ e^{\theta\bbj}\}$. Then $Q'$ is homeomorphic to a product of two 2-spheres.  We show $[Q']=2y\in H_4(\chi(F_2))=\ZZ\langle y\rangle$.
We   also construct an explicit embedded 2-sphere  $L'\subset \chi(F_2)$  representing the generator  of $H_2(\chi(F_2))$, and  show that the 4-dimensional abelian locus  represents  $4y\in  H_4(\chi(F_2))$ (this last fact also follows via the Narasimhan-Seshadri correspondence). The submanifolds $Q', L', \chi^\ab(F_2)$ intersect geometrically in $\chi(F_2)$  identically to how their homeomorphic and holomorphic counterparts,  
the smooth quadratic surface $Q$, a complex line $L$, and  the singular Kummer quartic surface $K\cong\TT^4/\{\pm 1\}$,  intersect in $\CC P^3$.

\subsection{Relationship to previous work} 
A proof that $\chi(F_2)$ is   a closed topological manifold, which avoids the Narasimhan-Seshadri correspondence and instead uses invariant theory in a symplectic setting,  was given by Huebschmann on pages 215-220 of
\cite{Hueb2}. His proof relies on parts of his earlier work  \cite{Hueb1} analyzing the local Poisson structure of $SU(2)$ character varieties of   compact oriented  surfaces.    
We provide a direct topological proof that $\chi(F_2)$ is a manifold, based on the result that $\chi(F_2)$ is the orbit space of an involution on the {\em traceless character variety} $\chi(S^2,6)$ of the 2-sphere with 6 marked points.

A proof that $\chi(F_2)$ is homeomorphic to $\CC P^3$ which avoids  the Narasimhan-Seshadri correspondence, but relies on the fact that $\chi(F_2)$ is a manifold,   was
given by Choi \cite{choi}.  
Choi's approach, and also the approach taken in   \cite{Hotal} which probes $\chi(F_2)$ using symplectic methods,  starts with a decomposition of a genus  two surface into two pairs of pants along three circles.  Modified Goldman flows (see \cite{goldman-invariant, jeffrey-weitsman,JW2} and Section \ref{sectiondiagram} below)  along these circles define a Hamiltonian $\TT^3$ action on a dense open subset of the smooth locus of $\chi^*(F_2)$,  which is then compared to the standard toric structure  $\CC P^3\to \Delta^3$.     By contrast, the proof   that $\chi(F_2)$ is homeomorphic to $\CC P^3$  presented here  starts from a decomposition of the genus two surface into two  punctured tori, and compares the character $\kappa\colon \chi(F_2)\to \RR$ of the separating curve  to a real-analytic map $\kappa_{\CC P^3}\colon \CC P^3\to \RR$  (\ref{Coltrane}). The map   $\kappa_{\CC P^3}$ is Morse-Bott with only two critical values, and decomposes $\CC P^3$ into the union of a tubular neighborhood of $\RR P^3$  and a tubular neighborhood of a smooth quadratic surface $Q$. A basic result of Wall \cite{Wall-classificationV} is used to show that any
smooth six manifold built out of two such tubular neighborhoods is diffeomorphic to $\CC P^3$.    Theorem \ref{thmA} thereby implies Theorem \ref{thmB}. 

\medskip

The results in this article are inspired by Boozer's article \cite{boozer}, in which he studies the traceless character variety $\chi(T^2,2)$ and its Lagrangian submanifolds by means of the character $\kappa'\colon \chi(T^2,2)\to \RR$ of a curve which separates $(T^2,2)$ into a pair of pants and a once-punctured torus.  Boozer proves $\chi(T^2,2)$ is homeomorphic to $S^2\times S^2$. One can prove, in a manner similar to   Theorem \ref{th4.1K6}, that $\chi(T^2,2)$ is doubly branched covered by the smooth manifold $\chi(S^2,5)$, known   to be  is diffeomorphic to $\CC P^2\# 5 \overline{\CC P}^2$ (\cite{KK,seidel}). 

The investigations of the present article, as well as Boozer's and  other articles including \cite{KaiS, Zuyi, CHKK,CHK,HK3}, aim to understand $SU(2)$ 
instanton-Floer theory from the ``bordered Floer theory" perspective \cite{LOT, BN,HRW, KWZ} with a focus on describing the Lagrangian correspondence
associated to    3-dimensional bordisms whose   character varieties have formal dimension greater than one. 

\subsection{Contents}   Section \ref{prelim}  reviews  the basics of character varieties, their smooth points,   and their tangent spaces.  The remainder of the article consists of three parts.   The first part, Sections
 \ref{sum6}, \ref{branch},  and  \ref{singpt},  contain  the  proof that $\chi(F_2)$ is a topological manifold.
The second part, Sections \ref{basics2}, \ref{princ}, and  \ref{pfthm1}   contain the proof that $\chi(F_2)$ is homeomorphic to $\CC P^3$. The last part, Section  \ref{lagrange}, contains examples  
 of Lagrangian immersions $\chi(Y)\to\chi(F_2)$, induced by   3-manifolds $Y$ with genus  two boundary, as well as constructions of even dimensional embedded cycles
representing   homology classes in $H_*(\chi(F_2))$.

 Apart from Section \ref{lagrange}, we have kept the required background to a minimum.  The proof that $\chi(F_2)$ is a manifold uses Morse theory and the equivariant tubular neighborhood theorem.   The basics of cohomology and characteristic classes are used in the proof that $\chi(F_2)$ is homeomorphic to $\CC P^3$.

 The examples in Section \ref{lagrange} rely  on the principle \cite{Atiyah1} that a compact 3-manifold with genus  two boundary (generically) determines a (stratified) Lagrangian immersion of its character variety into $\chi(F_2)$ \cite{herald1}.  In that section we   freely use the language of   symplectic geometry.

 Readers familiar with character varieties of surfaces and the Atiyah-Bott-Goldman symplectic form
   can quickly gain an overview of this article by reading, in succession, the statements of  Proposition \ref{dimension}, Theorem \ref{th4.1K6}, Theorem \ref{thm2.2}, Equations (\ref{kappa1}) and  (\ref{Coltrane}), Proposition \ref{P=NPforreal}, Theorems \ref{thm7} and \ref{thm8}, Corollary \ref{correspond}, and  the examples in Section \ref{ccm}.
%%%%%%%%%%
 \subsection{Acknowledgments.} Very special thanks to Dave Boozer, who provided us with the proof  of Theorem \ref{essfive}.
  The second author is indebted to Tom Mrowka for a discussion from which several of the ideas in this article originated, including  the example in Section \ref{w2}.   
 Thanks also to  Diarmuid Crowley, Jim Davis,  Jason DeVito, and Dan Ramras for helpful  discussions.
\section{Preliminaries}\label{prelim} 
This article focuses on   representations $\rho\colon \pi\to SU(2)$, which can have  one of only three orbit types. Proposition \ref{dimension} records various well-known properties of $SU(2)$ character varieties of {\em compact surfaces}. 
\subsection{Notation}

Denote by $\HH$ the quaternions
$A=a+b\bbi+c\bbj+d\bbk, ~a,b,c,d\in \RR$. Define the $\RR$-linear projection  $\Real\colon \HH\to \RR$ by $\Real(a+b\bbi+c\bbj+d\bbk)=a$, and set $su(2)=\ker \Real\colon \HH\to \RR$.  Denote by $\Ima={\rm Id}-\Real\colon \HH\to su(2)$.  {\em Quaternion conjugation} means the $\RR$-linear
map $$^-\colon\HH\to \HH,~\overline{a+b\bbi+c\bbj+d\bbk}=
a-b\bbi-c\bbj-d\bbk.
$$

We consider $\CC=\{a+b\bbi\}\subset \HH$ and note that 
 quaternion conjugation restricts to complex conjugation on $\CC$.
The quadratic form $\langle A,B\rangle=\Real\left(\overline{A}B\right)$ is  positive definite   on $\HH$.
\medskip

The notation $SU(2)$ and $S^3$ will be used interchangeably
to denote {\em the unit 3-sphere in $\HH$}, 
$$SU(2)=S^3=\left\{a+b\bbi+c\bbj+d\bbk\mid a^2+b^2+c^2+d^2=1\right\}.$$ 
This 3-sphere inherits a group structure from $\HH$, with $A^{-1}=\overline{A}$.  It contains the circle subgroup $\left\{a+ b\bbi\mid   a^2+b^2=1\right\}=\left\{ e^{\theta\bbi}\right\}= SU(2)\cap \CC=U(1)$.   
The notation $S^2$ will refer to {\em the unit 2-sphere in the Lie algebra $su(2)$}:  
$$S^2=\left\{b\bbi+c\bbj+d\bbk\mid  b^2+c^2+d^2=1\right\}=\{A\in SU(2)\mid  \Real(A)=0\}=SU(2)\cap su(2).$$
Unit quaternions in $S^2$ are called {\em traceless}; this is because   the isomorphism between $SU(2)$ and the unit quaternions satisfies ${\rm Tr}=2\Real$.

 We use frequently  use the  following  simple facts.
 \begin{itemize}
\item $S^2/\{\pm 1\}$ parametrizes the circle subgroups of $SU(2)$ via $\pm P \mapsto 
\left\{e^{\theta P}\mid  \theta\in \RR/2\pi\ZZ\right\}$.
\item $ \RR P^3,~SO(3), ~SU(2)/\{\pm 1\}$, the unit tangent bundle $UTS^2$ of $S^2$, 
 and the fixed point set of complex conjugation on $\CC P^3$ are diffeomorphic 3-manifolds.
\end{itemize}

 \subsection{Double branched covers}\label{branched}
 Given a smooth involution   $\sigma\colon X\to X$ on  a smooth  manifold 
whose fixed point set $X^{\rm fix}$ is a smooth  codimension two submanifold,
 the equivariant tubular neighborhood theorem \cite{bredon} implies that the orbit space 
 $Y:=X/\sigma$ is a topological manifold  
 which admits  at least one   smooth structure so that the orbit map $b\colon X\to Y$  is smooth and satisfies:
  \begin{enumerate}
\item The restriction $b\colon X^{\rm fix}\to b(X^{\rm fix})$ is a diffeomorphism,
\item the restriction $b\colon X\setminus X^{\rm fix}\to Y\setminus b(X^{\rm fix})$ is a smooth 2-fold covering space.
\end{enumerate}
We define a {\em smooth double branched cover  with branch set $X^{\rm fix}$} to be a smooth map $f\colon X\to Y$ of smooth  manifolds   such that the fibers of $f$ are the orbits of a smooth involution on $X$ with smooth codimension  two fixed point set $X^{\rm fix}$.
  \medskip

\subsection{$SU(2)$ character varieties of finitely presented groups} 

The basics concerning $SU(2)$  character varieties are carefully laid out in \cite{Akbulut-McCarthy}.   The literature on the spaces $\chi(F)$ and $\chi(S^n,k)$ is large.  For a few articles that take a similar approach to character varieties as the present article, see 
\cite{Klassen,Lin, Heusener-Kroll, JR, K6, jeffrey-weitsman} and references therein.  

\medskip

A  group presentation  \begin{equation} \label{fpgroup} G=\langle x_i, ~i=1,\dots,g\mid  w_j(x_1,\dots ,x_g), ~j=1,\dots ,r\rangle\end{equation} determines a  conjugation-invariant    map $$R\colon SU(2)^g\to SU(2)^r, ~ R(A_1,\dots, A_g)=\big(w_1(A_1,\dots, A_g),\dots, w_r(A_1,\dots, A_g)\big).$$
There is a natural bijection  $$\Hom(G,SU(2))\cong R^{-1}(1,\dots,1)$$ given by evaluation on the generators $x_i$.   Elements of $\Hom(G,SU(2))$ are called {$SU(2)$ \em representations of $G$}.
Topologize $\Hom(G,SU(2))$  as a subspace of $SU(2)^g\subset \RR^{4g}$.

 The orbit space, $R^{-1}(1) /\conju$ is called {\em the $SU(2)$ character variety of $G$}  and is denoted by $\chi(G)$.  It is straightforward to check that the spaces $\Hom(G,SU(2))$ and  $\chi(G)$ are, up to homeomorphism,
  independent of the choice of the presentation of $G$. We   abuse notation slightly and typically  write $\rho$ for both a representation and its conjugacy class when
it is clear from context which is meant. If we need to distinguish the two, we write $[\rho]$ for conjugacy class.

 The space $\chi(G)$  has a partition by orbit type for the conjugation action
 $$\chi(G)=\chi^*(G)\sqcup \chi^\anc(G)\sqcup \chi^\cen(G)$$
 where $\chi^*(G)$ denotes the conjugacy classes of representations with non-abelian image  (equivalently with stabilizer the center $\{\pm 1\}$), $\chi^\anc(G)$ those with abelian but non-central image (equivalently, with stabilizer  $U(1)$), and $\chi^\cen(G)$ those with central image (and stabilizer $SU(2)$).  
Set $\chi^\ab(G)=\chi^\anc(G)\sqcup\chi^\cen(G)$ to be those conjugacy classes with abelian image.
Then $\chi(G)$ has a filtration
$$\chi^\cen(G)\subset \chi^\ab(G)\subset \chi(G).$$

\medskip

Define {\em traceless} character varieties in the following way.  If $\gamma\in G$, define the {\em character} of $\gamma $ to be the function $\kappa_\gamma\colon \chi(G)\to [-1,1]$ given by $$\kappa_\gamma(\rho)=\Real(\rho(\gamma)).$$ 
Given an ordered list $S=(\gamma_1,\gamma_2,\dots, \gamma_n)\in G^n$, define the  map
$$T_{S}\colon SU(2)^g\to \RR^k \text{ by } T_{S}(\rho)=(\kappa_{\gamma_1}, \dots,\kappa_{\gamma_n})(\rho)=(\Real(\rho(\gamma_1)), \dots, \Real(\rho(\gamma_n))).$$
The map $T_S$ can be combined with $R$ to produce 
$$\widetilde{R}_S:=R\times T_{S}\colon SU(2)^g\to SU(2)^r\times \RR^n.$$

Representations in the subset $\widetilde{R}_S^{-1}(1,0)\subset R^{-1}(1)=\Hom(G,SU(2))$ are called {\em traceless} representations, and  the orbit space, $\widetilde{R}_S^{-1}(1,0) /\conju$ is denoted by $\chi(G,S)$  and called the  {\em traceless character variety} of $(G,S).$ When $S$ is non-empty, 
$\chi^\cen(G,S)$ is empty, and hence $\chi(G,S)=\chi^*(G,S)\sqcup \chi^\anc(G,S)$.

\medskip

 If $(X,Y)$ is a smooth codimension  two pair of compact manifolds, a {\em meridian} of $(X,Y)$ is, by definition,  a based loop in $X\setminus Y$ bounding a 2-disc in $X$ which intersects $Y$  in precisely one point, transversely. Choosing a finite set $S$ of meridians representing all meridian conjugacy classes,  define $$\chi(X,Y):=\chi(\pi_1(X\setminus Y),S).$$   In the special case when $X$ is a connected 2-manifold
 and $Y$ is a finite set of $n$ points, we write $\chi(X,n)$.

\subsection{Smooth  points  and tangent spaces of  $ SU(2)$ character varieties.} \label{Se2.2} 

The following differential-topological definition of smoothness suffices for this article.   A comprehensive discussion of smooth points of algebraic sets can be found in Chapter 2 of  Milnor's book \cite{milnor2}. 

The group $SU(2)\subset \HH=\RR^4$ is defined by the polynomial equation $a^2+b^2+c^2+d^2=1$ and $R\colon SU(2)^g\to SU(2)^r$ extends to a polynomial map   $\RR^{4g}\to \RR^{4r}$. Therefore 
$\Hom(G,SU(2))\subset \RR^{4g}$ is a real affine algebraic set.

 As   orbit spaces of a compact group action on a real algebraic set, $\chi(G)$  and its subset $\chi(G,S)$ are compact   semi-algebraic    sets (defined by polynomial identities and polynomial  {\em inequalities} \cite{BCR}).   The spaces $\Hom(G,SU(2))$ and  $\chi(G)$ are, up to polynomial isomorphism,  independent of the choice of the presentation of $G$ (see e.g. \cite{Culler-Shalen}).
Hence we use the (somewhat imprecise) terminology character {\em varieties}. 

\begin{df}  If $G$ is presented as in (\ref{fpgroup}), 
define $\rho\in \Hom(G,SU(2))\subset \RR^{4g}$ to be a {\em smooth point} if there exists a neighborhood $U\subset \RR^{4g}$ of $\rho$  and a smooth submersion $f\colon U\to \RR^k$
for some $k$ so that $U\cap \Hom(G,SU(2))=f^{-1}(0).$ Hence $\Hom(G,SU(2))$ is a smooth $4g-k$ manifold near a smooth point.

Define $[\rho]\in \chi(G)$ to be a {\em smooth point} provided one of its  representatives $\rho\in \Hom(G,SU(2)))$ 
is a smooth point  and   the neighborhood $U$  can be chosen so that the orbit type is constant on $U$. By the equivariant tubular neighborhood theorem, $\chi(G)$ is a smooth manifold of dimension $4g-k-\dim({\rm Stab}(\rho))$ near a smooth point $\rho$.  We call the subset of smooth points the {\em smooth locus of $\chi(G)$}.
\end{df}

  Weil's theorem \cite{weil} gives a natural identification of  the  tangent space  at a smooth point $\rho$ in $ \chi(G)$ with $H^1(G; su(2)_{{\rm ad} \rho})$,  where $su(2)_{{\rm ad} \rho}$ denotes the   local coefficients given by $ad\circ \rho \colon  G\to SU(2)\to GL(su(2))$ (see e.g. \cite[Chapter 5]{DaK}). 
The differential at $\rho$  of the map $h^*\colon \chi(G_2)\to \chi(G_1)$ induced by a group homomorphism $h\colon G_1\to G_2$, when $\rho$ and $h^*(\rho)$ are smooth points, is canonically identified with the induced map on cohomology 
$H^1(G_2; su(2)_{{\rm ad} \rho})\to H^1(G_1; su(2)_{{\rm ad} h^*(\rho)})$.
At {\em any} $\rho\in\chi(G)$, the group $H^1(G; su(2)_{{\rm ad} \rho})$   is called the {\em Zariski tangent space  of $\chi(G)$ at $\rho$}.     Recall that the first cohomology of a space $X$ and the first cohomology of its fundamental group $G=\pi_1(X)$ are naturally isomorphic.

\subsection{Surfaces and the Atiyah-Bott-Goldman form}

The $SU(2)$ character variety of a compact surface  is simpler than the character variety of an arbitrary  group  because its non-abelian locus is cut out transversely by $R$ or $\widetilde R_S$ for the usual   presentation of the surface group, and, with the few exceptions of surfaces with abelian fundamental groups,  its non-abelian locus coincides with its smooth locus.  The following well-known  proposition can be proven directly  by
calculating $dR$ and $d\widetilde R_S$, or by calculating Zariski tangent spaces, an easy task using the Fox calculus \cite{fox}; see e.g.~ \cite{Akbulut-McCarthy, Heusener-Kroll, JR, K6}.

\begin{proposition}\label{dimension}  Let $F$ be a closed and connected surface of genus $g$ and $S\subset F$ a finite set of $n$ points.   Assume that  $6g+2n> 6$. Then
the subspace of non-abelian conjugacy classes, 
 $\chi^*(F,n)\subset\chi(F,n)$,  coincides with the smooth locus of $\chi(F,n)$, and  
is a non-empty smooth manifold of dimension  $6g+2n-6$.  
\end{proposition}

\subsubsection{Symplectic structure} In the case of of a closed oriented surface $F$, the cup product 
$$H^1(F; su(2)_{{\rm ad} \rho})\times H^1(F; su(2)_{{\rm ad} \rho})\to H^2(F; \RR)\cong \RR,$$  using the bilinear form $\Real(\bar A B)$ on $su(2)$,   
is skew-symmetric and non-degenerate by Poincar\'e duality.  The theorem of Atiyah-Bott and Goldman
asserts that, using Weil's identification of tangent spaces, this 2-form on $\chi^*(F)$ restricts to a {\em closed, smooth}, and hence symplectic, form on the smooth non-abelian locus $\chi^*(F)$ \cite{AB, Goldman}.
More generally, the Zariski tangent space of a traceless character variety $\chi(F,n)$  has a cohomological definition and     the cup product defines a symplectic form \cite{Jeffreyetal, CHK} on $\chi^*(F,n)$.

\subsection{The link of a point in a character variety}
 Evaluation on the $g$ generators of a finitely presented group $G$ embeds $\Hom(G, SU(2))$ into  $SU(2)^g$, equivariantly  under the conjugation action. 
Fix $\rho\in \Hom(G, SU(2)$ and let
$O_\rho\subset SU(2)^g$ denote the orbit of $\rho$,  a smooth invariant submanifold.

The equivariant tubular neighborhood theorem provides 
 a   normal vector bundle $N \to  O_\rho$, equipped with a metric and  an   action $SU(2)\to {\rm Isom}(N )$, and
 for some $\ep>0$, a smooth equivariant embedding  of  the closed $\ep$ disk bundle $N (\ep)$ of $N $ into $SU(2)^g$,
with image  a closed $SU(2)$ invariant tubular neighborhood $U_\ep$.  The intersection $U_\ep\cap  \Hom(G, SU(2))$ is invariant.
Define  
$${\rm Lk}(\rho,\ep; \chi(G)):= \big(\partial U_\ep\cap \Hom(G, SU(2))\big)/\conju \subset \chi(G)\subset SU(2)^g/\conju.$$ 
If there exists an $\ep_0>0$ so that the homeomorphism type of ${\rm Lk}(\rho,\ep; \chi(G))$ is independent of
$0<\ep\leq \ep_0$, define ${\rm Lk}(\rho; \chi(G))$ to be the homeomorphism type of ${\rm Lk}(\rho,\ep_0; \chi(G)).$    If ${\rm Lk}(\rho; \chi(G))$ is defined and homeomorphic to an $(\ell-1)$-sphere,
then $\chi(G)$ is a topological $\ell$-dimensional manifold near $\rho$.

 \subsection{The genus two surface}\label{g2a}

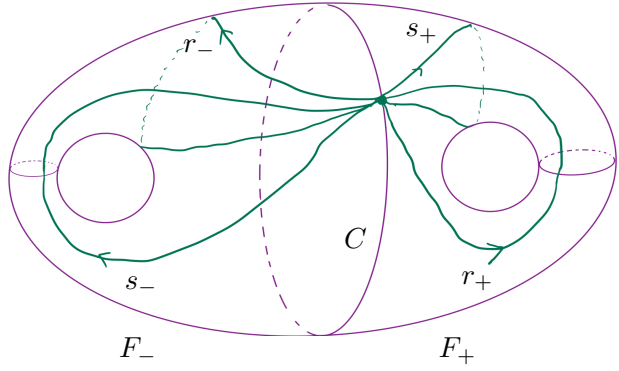
\begin{figure}[ht!]
\labellist
\pinlabel $r_-$ at 225 485
\pinlabel $s_+$ at 420 500
\pinlabel $s_-$ [bl] at 155 245
\pinlabel $r_+$ [bl] at 450 250
\pinlabel $C$ [bl] at 347 290
\pinlabel $F_-$ [bl] at 150 180
\pinlabel $F_+$ [bl] at 430 180
\endlabellist
\centering
\includegraphics[width=3.3in, height=1.9in]{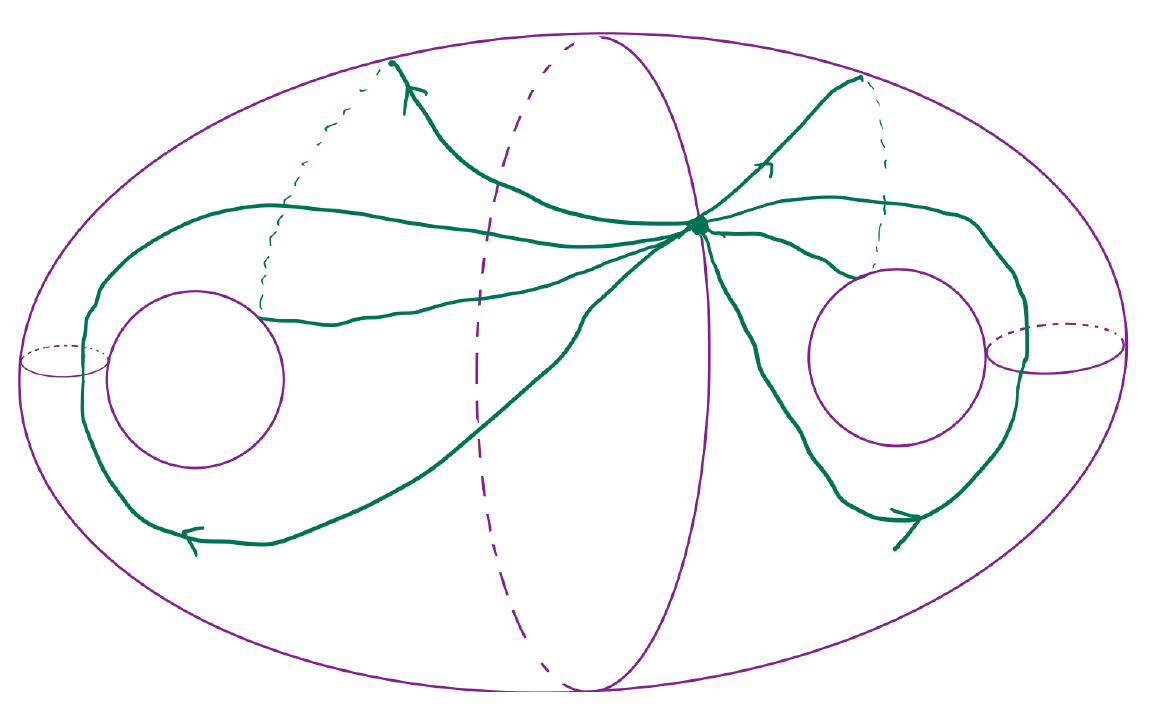}
 \caption{\label{fig0fig} The genus  two surface $F_2$}
\end{figure}

 The fundamental group of the closed oriented genus two surface $F_2$ has the standard presentation 
$$\pi_1(F_2)=\langle r_-,s_-,r_+,s_+\mid   [r_-, s_-][r_+,s_+]=1\rangle.$$
 Figure \ref{fig0fig}   indicates  the  four loops, $r_-,s_-,r_+,s_+$, representing the generators. Also shown is  a  curve $C$ separating $F_2$ into punctured tori $F_-$ and $F_+$, with $C=[r_-,s_-]=[s_+,r_+]$. 
 
The character variety $\chi(F_2)$ has non-abelian locus $\chi^*(F_2)$ a smooth 6-manifold
(Proposition \ref{dimension}). 
 The abelian locus $\chi^\ab(F_2)$   consists of representations conjugate to 
$$ r_-\mapsto e^{\alpha_-\bbi},~ s_-\mapsto e^{\beta_-\bbi},~ r_+\mapsto e^{\alpha_+\bbi},~ s_+\mapsto e^{\beta_+\bbi}.$$  From this one  sees that  
$$\chi^\ab(F_2)=\Hom\left(\ZZ^4, U(1)\right)/W= \TT^4/\{\pm 1\},$$
where $W=\{\pm 1\}$ denotes the Weyl group of $SU(2)$.  Here $-1$ acts on $\TT^4=(\RR/2\pi\ZZ)^4$ by negation, $(\alpha_-,\beta_-,\alpha_+,\beta_+)\mapsto -(\alpha_-,\beta_-,\alpha_+,\beta_+)$.   The sixteen fixed points of this involution, namely  
the homomorphisms $\pi_1(F_2)\to \{\pm 1\}$, form  the central locus $\chi^\cen(F_2)$.  The equivariant tubular neighborhood theorem, applied to $(\TT^4, \pm 1)$,  implies that $\chi^\anc(F_2)\subset \chi^\ab(F_2)$ is a smooth 4-manifold and identifies the link of $s\in \chi^\cen(F_2)$  in $\chi^\ab(F_2)$ as a cone on $\RR P^3$. 

\medskip

The link  of each $s\in \chi^\cen(F_2)$  in $\chi(F_2)$ is identified with $S^5$ in Theorem \ref{thm2.2} below.

\subsection{Six-punctured 2-sphere}

The six punctured 2-sphere $S^2\setminus \{a_1,\dots, a_6\}$ has fundamental group
$$\pi_1(S^2\setminus \{a_1,\dots, a_6\})=\langle x_1,x_2,x_3,x_4,x_5,x_6\mid   x_1x_2x_3x_4x_5x_6=1\rangle.$$ The generator  $x_i$ is the meridian encircling the
 $i$th puncture $a_i$.
Its traceless character variety, 
$$\chi(S^2,6)=\left\{ \rho\colon \pi_1\left(S^2\setminus \{ a_i\}\right)\to SU(2) \mid    \Real(\rho(x_i))=0\right \}/_\conju ,$$
has non-abelian locus $\chi^*(S^2,6)$ a smooth 6-manifold  (Proposition \ref{dimension}).  
  The abelian non-central locus, $\chi^\anc(S^2,6)$, contains sixteen points, namely the conjugacy classes  of 
$$\rho\colon  x_1\mapsto\bbi, x_2\mapsto\pm \bbi, \dots, x_5\mapsto\pm \bbi,~ x_6\mapsto \overline{ x_1 x_2 x_3 x_4 x_5}.$$
The central stratum  $\chi^\cen(S^2,6)$ is empty, and so 
$\chi^\anc(S^2,6)=\chi^\ab(S^2,6)$. In summary, $\chi(S^2,6)$ is a smooth 6-manifold away from these sixteen abelian points.
The link of each $s\in \chi^\ab(S^2,6)$  in $\chi(S^2,6)$ is identified with  the unit tangent bundle of $S^3$ in Theorem \ref{thm2.2} below.

 \subsection{The free groups of rank one and two}\label{listofch}  The character varieties of the free groups of rank one and two  are not covered by Proposition \ref{dimension}, but their  basic structure is easily understood.  

First,  $\chi^*(\ZZ)$ is empty, and $\chi(\ZZ)=\chi^\ab(\ZZ)=SU(2)/\conju$.  The map $\kappa\colon \chi(\ZZ)\to [-1,1]$ given by $\rho\mapsto \Real(\rho(1))$  is a homeomorphism which sends $\chi^\cen(\ZZ)$ to the endpoints. Since 
 $\chi^\anc(\ZZ)$ is the   space of principal orbits, it is a smooth open interval, and  $\kappa\colon \chi^\anc(\ZZ)\to (-1,1)$ a diffeomorphism.

For the free group of rank 2, $\chi(\langle r,s\rangle)=(SU(2)\times SU(2))/\conju$. The non-abelian locus $\chi^*(\langle r,s\rangle)$  coincides with the  space of principal orbits of the conjugation action, and hence is a smooth 3-manifold. 
See Section \ref{GP}   for a complete description of $\chi^*(\langle r,s\rangle)$.

 \section{$\chi(F_2)\setminus \chi^\cen(F_2)=\chi^*(F_2)\sqcup \chi^\anc(F_2)$ is   a  6-manifold.}\label{sum6}
 
This section and Appendix \ref{APPB} provide a  proof that  $\chi(F_2) \setminus \chi^\cen(F_2)$   is homeomorphic to a smooth 6-manifold, using a double branched cover  $p^*\colon \chi^*(S^2,6)\to \chi(F^2)\setminus \chi^\cen(F_2)$,
branched over   $\chi^\anc(F_2)$.

\subsection{An involution on $\chi(S^2,6)$ with orbit space $\chi(F_2)$} \label{Gtop}
The variety $\chi(S^2,6)$ admits the involution
\begin{equation}\label{nu}\nu\colon \chi(S^2,6)\to \chi(S^2,6),~(\nu\cdot \rho)(x_i)=-\rho(x_i), ~ i=1,\dots ,6.\end{equation}
Consider the map $\TT^4\to \chi(S^2,6)$ which sends $(\alpha_2, \dots, \alpha_5)$ to the conjugacy class of the representation 
 $$x_1\mapsto \bbi, ~x_i\mapsto e^{\alpha_i \bbk}\bbi\text{ for } i=2,3,4,5,\text{ and } x_6\mapsto -\bbi e^{(\alpha_2-\alpha_3+\alpha_4-\alpha_5)\bbk} .$$ 
It is easy to check that the image of this map is the fixed point set of $\nu$, which we denote by $\chi^{\rm fix}(S^2,6)$.  This map   is invariant under $(\alpha_2, \dots, \alpha_5)\mapsto (-\alpha_2, \dots, -\alpha_5)$, and induces a homeomorphism
      $$\chi^{\rm fix}(S^2,6)\cong \TT^4/\{\pm 1\}.$$ 
Since the involution $\nu$ acting on $\chi^*(S^2,6)$ is smooth,  $\chi^{\rm fix}(S^2,6)\setminus \chi^\ab(S^2,6)$   is a smooth   codimension  two submanifold of $\chi^*(S^2,6)$.

 Define  
 \begin{equation}\label{themapp} p^*\colon \chi(S^2,6)\to \chi(F_2)\end{equation}
 by 
$$p^*(\rho)\colon r_-\mapsto \rho(x_1x_2),~  
s_-\mapsto \rho(x_3^{-1}x_2^{-1}),~r_+\mapsto \rho(x_4x_5),~ s_+\mapsto \rho(x_6^{-1}x_5^{-1}).
$$
\begin{theorem}[{\cite[ Theorem 4.1]{K6}}]\label{th4.1K6}
The map $p^*$ induces a homeomorphism  
$$\chi(S^2,6)/\nu \to \chi(F_2).$$ The involution $\nu$ has codimension two fixed point set $\chi^{\rm fix}(S^2,6)$ homeomorphic to $\TT^4/\{\pm1\}$. In addition,  $p^*(\chi^\ab(S^2,6))=\chi^\cen(F_2)$ and 
 $p^*(\chi^{\rm fix}(S^2,6))= \chi^\ab(F_2)$.
\end{theorem}
For the reader's convenience, we include a short proof of Theorem \ref{th4.1K6} in Appendix \ref{APPB}.

 \begin{corollary}\label{ascor}   $\chi(F_2)\setminus  \chi^\cen(F_2)=\chi^*(F_2)\sqcup \chi^\ab(F_2)$  is   a  6-manifold. 
\end{corollary}
 \begin{proof}   
 The surjection  $$p^*\colon \chi^*(S^2,6) \to \chi(F_2)\setminus \chi^\cen(F_2)$$
 induces a homeomorphism 
 $$\overline{p}^*\colon \chi^*(S^2,6)/\nu \to \chi(F_2)\setminus \chi^\cen(F_2)$$
 since the fibers of $p^*$ coincide with the orbits of $\nu$.  Because $\nu$ is a smooth involution on the smooth manifold $\chi^*(S^2,6)$,   the discussion of Section \ref{branched} implies that $\chi^*(S^2,6)/\nu=\chi(F_2)\setminus  \chi^\cen(F_2)$ is   a  6-manifold. 
 \end{proof}

An explanation of the geometric origin of the formula for $p^*$ can be found in \cite{K6}. In  that article,  the map $p^*$ is defined by $p^*=(-1)^c\cdot b^*$, where $b^*$ is induced by a double branched cover $b\colon F_2\to S^2$, branched over six points, and $c\in H^1(F_2\setminus 6;\ZZ/2)$ is the cocycle which evaluates to  one one each circle enclosing one of the six branch points in $F_2$.    
 
 \section{A double branched cover $UTS^n\to S^{2n-1}$.}\label{branch}
 
 This and the following section contain the proof that each of the sixteen points in $\chi^\cen(F_2)$
 has a neighborhood in $\chi(F_2)$ homeomorphic to  $\RR^6$.  
 In this section, we prove that the differential  of the reflection on $S^n$ through its equator, acting on the unit tangent bundle of $S^n$, has  fixed point set the unit tangent bundle of $S^{n-1}$ and orbit space $S^{2n-1}$.  While we only need the $n=3$ case for this article, we include  the proof for general $n$ since it is no   harder than $n=3$.

 Define $UTS^n\subset \RR^{n+1}\times \RR^{n+1}$ to be the unit tangent bundle of $S^n$. Explicitly
 $$UTS^n=\left\{ (w_1,w_2)\in  \RR^{n+1} \times \RR^{n+1}\mid  w_1\cdot w_2=0,~ |w_1|^2=1,~ |w_2|^2=1\right\}.$$
 Denote by ${\rm proj}\colon \RR^{n+1}\to \RR^n$  the projection to the first $n$ components. 
 Given $(w_1,w_2)\in UTS^{n}$,
 $|{\rm proj}(w_1)|^2+|{\rm proj}(w_2)|^2\ne 0$, since $w_1$ and $w_2$ are linearly independent. Hence the function
\begin{equation*} p\colon UTS^n\to S^{2n-1},~
 p(w_1,w_2)=\frac{({\rm proj}(w_1),{\rm proj}(w_2))}{\left(|{\rm proj}(w_1)|^2+|{\rm proj}(w_2)|^2\right)^{\frac 1 2}}
 \end{equation*}
 is defined and smooth.

\begin{theorem}\label{essfive} Let    $\tau \colon  UTS^n\to UTS^n$ be the involution induced by the reflection through the hyperplane $\RR^n\oplus \{ 0\}\subset \RR^{n+1}$. Then the fixed set of $\tau$ is $UTS^{n-1}$,  $p$ induces a homeomorphism $UTS^{n+1}/\tau\cong S^{2n-1}$,  and   $p$ is a smooth double branched cover.  \end{theorem}
\begin{proof}  Write $\RR^{n+1}$ as $\RR^n\oplus \RR e$; then  $\tau(v+te)=v-re$. The differential of $\tau$ acting on $UTS^n\subset \RR^{n+1}\oplus \RR^{n+1}$, which we also denote by $\tau$, is given by $$\tau(v_1+r_1e,v_2+r_2 e) =(v_1-r_1e,v_2-r_2e) .$$
The fixed set of $\tau$ is  $UTS^{n-1}$. 
Referring to Section \ref{branched}, it suffices to show that the fibers of $p\colon UTS^n\to S^{2n-1}$ coincide with the orbits of $\tau$.  That $p\circ \tau=p$ is clear.  Hence the fibers of $p$ are (possibly empty) unions of $\tau$ orbits.

\medskip

Fix $ v=(v_1,v_2)\in S^{2n-1}\subset \RR^n\oplus \RR^n$.   Hence $|v_1|^2+|v_2|^2=1.$ 
 If $v_1=0$ or $v_2=0$, it is easy to see that
 $p^{-1}(v)$ equals $\{(\pm e, v_2)\}$ or 
 $p^{-1}(v)=\{(v_1,\pm e)\}$, in either case precisely one $\tau$ orbit.
Assume for the rest of the proof that  both $v_i$ are non-zero, and  let $\gamma$ denote the angle between $v_1$ and $v_2$.

Denote by $S_v\subset \RR^3$ the set of triples $(\alpha,\beta,\lambda)\in \RR^3$ satisfying
\begin{equation}\label{three}
\lambda^2 (v_1\cdot v_2)+\al\be=0,~\alpha^2 =1-\lambda^2|v_1|^2,~
\be^2=1-\lambda^2|v_2|^2, ~ |\alpha|\leq 1,~|\beta|\leq 1, ~0\leq \lambda\leq\sqrt{2}.
\end{equation}
Then there is a bijection of $S_v\to p^{-1}(v)$ given by
$$(\alpha,\beta,\lambda)\mapsto
(\lambda v_1+ \alpha e, \lambda v_2+ \beta e),$$
and $\tau$ pulls back to $(\alpha,\beta,\lambda)\mapsto (-\alpha,-\beta,\lambda)$ on $S_v$.
Any  $(\alpha,\be,\lambda)\in S_v$ satisfies 
$$\lambda^2=2-(\alpha^2+\beta^2) \text{  and }
 \lambda^4  (v_1\cdot v_2)^2=\alpha^2\beta^2=
(1-\lambda^2|v_1|^2)(1-\lambda^2|v_2|^2).
$$
The second equality can be rewritten
$$  \lambda^4 A-\lambda^2+1=0$$ where
$$
A=A(v_1,v_2)=|v_1|^2|v_2|^2-(v_1\cdot v_2)^2=|v_1|^2|v_2|^2 \sin^2\gamma.
$$
  Since $|v_1|^2|v_2|^2=|v_1|^2(1-|v_1|^2)$
and the function $f(x)=x^2(1-x^2)$ is non-negative on $[0,1]$ with maximum value $\tfrac 1 4$, it follows that $0\leq A\leq \tfrac 1  4.$

\medskip
 
If $A=0$, then $\lambda=1$ and $v_1=\ep v_2$ for some $\ep\in \{\pm 1\}$.  
This implies  $$(\al,\be,\lambda) =(\pm \sqrt{1-|v_1|^2}, \pm \sqrt{1-|v_2|^2},1)=(\ep_1 |v_2|, \ep_2 |v_1|,1) $$
for some $\ep_1,\ep_2\in\{\pm 1\}$.
The first equation in (\ref{three}) implies that 
$$\ep_1\ep_2|v_1| |v_2|=\alpha\beta=-\lambda^2(v_1\cdot v_2)=-(v_1\cdot v_2)=-\ep|v_2|^2.$$
  and so $\ep_1\ep_2=-\ep$.  Hence,  $S_v=\{(\ep_1|v_2|,\ep_1(-\ep |v_1|),1)$ and so
  $p^{-1}(v)$ consists of precisely one $\tau$ orbit.

\medskip

Assume instead that  $A> 0$.  The quadratic formula implies that any solution 
  $(\alpha, \beta, \lambda)$ satisfies 
  $$\lambda^2=\frac{1+\ep_3\sqrt{1-4A}}{2A}
  $$ for some $\ep_3\in\{\pm 1\}$.
Since $\tfrac{1+ \sqrt{1-  4A  }}{ 2A}\ge \tfrac{1}{ 2A}\ge 2$ and 
$\lambda^2\leq 2$, it follows that either $A=\tfrac 1 4$ or $\ep_3=-1$.
If $A=\tfrac 1 4$, then the choice of $\ep_3$ does not matter and so $\lambda^2$ satisfies this equation with  $\ep_3=-1$.

\begin{lemma} \label{Cbs}
 $\tfrac{1-  \sqrt{1-  4A  }}{2A}|v_1|^2\leq 1$ and $\tfrac{1-  \sqrt{1-  4A  }}{2A}|v_2|^2\leq1$. 
\end{lemma}
\begin{proof}  Since $A=|v_1|^2|v_2|^2 \sin^2\gamma$ and $|v_1|^2|v_2|^2=|v_1|^2(1-|v_1|^2)$, it follows that
$$
-\sin^2\gamma|v_1|^2|v_2|^2\ge- |v_1|^2|v_2|^2=-|v_1|^2(1-|v_1|^2)
$$
which implies $$
1-4A=1-4|v_1|^2|v_2|^2 \sin^2\gamma\ge 1-4|v_1|^2+ 4 |v_1|^4=(1-2|v_1|^2)^2
$$
and therefore
$\sqrt{1-4A}\ge  2|v_1|^2-1.
$
Hence 
$1+\sqrt{1-4A}\ge 2|v_1|^2 
$ 
and so
$$
2(1+\sqrt{1-4A})|v_2|^2\sin^2\gamma\ge 4 |v_1|^2 |v_2|^2\sin^2\gamma=4A=(1+\sqrt{1-4A})(1-\sqrt{1-4A}).
$$
Since $(1+\sqrt{1-4A})\ge 1$, this implies
$$1-\sqrt{1-4A}\leq 2|v_2|^2\sin^2\gamma = \frac{2A}{|v_1|^2}, $$
proving the first assertion. The second assertion is proven similarly, reversing the 
roles of $v_1$ and $v_2$.
\end{proof}

Lemma \ref{Cbs}  and the three equations displayed in (\ref{three}) imply
\begin{equation*} \label{sols}
 S_v= 
\left\{
\left(
\pm \sqrt{ 1-\tfrac{1-  \sqrt{1-  4A  }}{2A}|v_1|^2},
\pm \sqrt{ 1-\tfrac{1-  \sqrt{1-  4A  }}{2A}|v_2|^2}
, 
\sqrt{\tfrac{1-  \sqrt{1-  4A  }}{ 2A}}\right)
\right\}.
\end{equation*} 
Therefore $S_v$ is non-empty and consists of at most four points, according to  
 the two possible signs for $\alpha$ and $\beta$.     If $v_1\cdot v_2\neq 0$, 
 the equation $\alpha\beta=-\lambda^2  (v_1\cdot v_2)$
 shows that the sign of $\alpha$ is determined by that of $\beta$. Thus, $p^{-1}(v)$ consists of a single orbit.
 On the other hand, if $v_1\cdot v_2=0$, then  at least one of $\alpha,\beta$ equals zero, so $S_v$  contains  one or two points, and the fiber again consists of a single orbit. \end{proof}

The fixed point set of $\tau$, $UTS^{n-1}$, is embedded by $p$ with image 
\begin{equation}\label{fix4}
\{v=(v_1, v_2)\in S^{2n-1}\subset \RR^n\oplus\RR^n  \mid   |v_1|=\tfrac{1}{\sqrt{2}},~ |v_2|=\tfrac{1}{\sqrt{2}}, ~v_1\cdot v_2=0\}.\end{equation}

  \section{$\chi(F_2)$ is a topological manifold.}\label{singpt}
  
  In this section we prove that $\chi(F_2)$ is a manifold near $\chi^{cen}(F_2).$ We show that the link of each central representation is homeomorphic to $S^5$  by identifying its double branched cover equivariantly with $(UTS^3,\tau)$ for
  $\tau$ the differential of reflection of $S^3$ through a hyperplane,  and applying Theorem \ref{essfive}.  
 The proof, given in Section \ref{5.2}, involves  identifying various   orbit  spaces  of $O(2)$ actions  by making equivariant identifications of prequotients.  For clarity, we set the stage by listing the $O(2)$-spaces that occur in the argument in Section \ref{5.1}.

\subsection{Actions of $O(2)$}   \label{5.1} Consider $O(2)$, expressed as the  semidirect product of the group $S^1$  and $\ZZ/2$:
$$O(2)=S^1\ltimes \{1,R\} \text{, where }  R^2=1 \text{ and } Re^{\theta\bbi}R^{-1}=e^{-\theta\bbi}.$$   
\begin{itemize}
\item  Let $O(2)$ act  on $(S^2)^4$ by 
\begin{align} \label{act1} e^{\theta\bbi} \cdot (X_1,X_2,X_3,X_4)&=\left(e^{\tfrac{\theta}{2}\bbi}X_1e^{-\tfrac{\theta}{2}\bbi},e^{\tfrac{\theta}{2}\bbi}X_2e^{-\tfrac{\theta}{2}\bbi},e^{\tfrac{\theta}{2}\bbi}X_3e^{-\tfrac{\theta}{2}\bbi},e^{\tfrac{\theta}{2}\bbi}X_4e^{-\tfrac{\theta}{2}\bbi}\right),\\
 R\cdot (X_1,X_2,X_3,X_4)&=\left(-\bbk X_1\overline{ \bbk},-\bbk X_2 \overline{ \bbk},-\bbk X_3 \overline{ \bbk},-\bbk X_4 \overline{ \bbk}\right)\nonumber.\end{align}
\item Let $O(2)$ act on  $\CC^4$ by rotation and complex conjugation: \begin{equation}\label{act2} e^{\theta\bbi}\cdot(z_1,z_2,z_3,z_4)=\left(e^{\theta\bbi} z_1,e^{\theta\bbi} z_2,e^{\theta\bbi} z_3,e^{\theta\bbi} z_4\right) , ~R\cdot (z_1,z_2,z_3,z_4)=\left(\overline{z_1},\overline{z_2},\overline{z_3},\overline{z_4}\right).\end{equation}
\item Let  $O(2)$ act on $\HH\times \HH$ by 
\begin{equation} \label{O2HH} e^{\theta \bbi}\cdot  (A,B) = (e^{\theta \bbi} A, e^{\theta \bbi } B) \text{ and } R\cdot (A,B)=(\bbj B \overline{\bbj}, \bbj A \overline{\bbj}).\end{equation} 
This action leaves ${\rm Cone}(S^3\times S^3)\subset \HH \times \HH$ invariant.
 We mean here the infinite cone, that is, all non-negative real multiples of points in $S^3\times S^3$.

\item Extend the involution $(\RR, R\cdot x= -x)$ to  an $O(2)$ action by letting
$S^1$ act trivially.
\end{itemize}

\subsection{The link of a central representation in $\chi(F_2)$ is $S^5$}\label{5.2} 
\begin{theorem} \label{thm2.2}
For each $s\in \chi^{\rm ab}(S^2,6)$, 
there exists a  $\ZZ/2$-equivariant diffeomorphism
$$ ({\rm Lk}(s;\chi(S^2,6)),\nu)  \cong  \big(UTS^3, \tau),$$  where 
 $\nu$ is the restriction  of the involution $\nu\colon   \chi(S^2,6) \to  \chi(S^2,6)$ defined in (\ref{nu}) to ${\rm Lk}(s;\chi(S^2,6))$, and $\tau$ is the differential of reflection of $S^3$ through a hyperplane. 
Hence,  $ p^*(s)\in \chi^{\rm cen}(F_2)$ has a neighborhood in $\chi(F_2)$ homeomorphic to a cone on $S^5$, so $\chi(F_2)$ is   a   closed (topological) 6-manifold.  
\end{theorem}

\begin{proof} 
 Define $U\colon (S^2)^4\to \RR$ by
 $$U(  X_1,X_2,X_3,X_4)=\Real\left( \overline{X_1X_2X_3X_4\bbi} \right).$$
Then $U$ is $O(2)$-equivariant with respect to (\ref{act1}) and so the  subspace $U^{-1}(0)$
is $O(2)$-invariant. 
 
Define a map $I\colon (S^2)^4\to \chi(\langle x_1, \dots, x_6\rangle)$ by
$$I(X_1,X_2,X_3,X_4)\colon \left[ x_1\mapsto X_1, ~x_2\mapsto X_2,~x_3\mapsto X_3,~x_4\mapsto X_4, ~x_5\mapsto\bbi,
~x_6\mapsto \overline{X_1X_2X_3X_4\bbi}~\right].$$
 \begin{lemma}\label{Dinv}
 The map $I$ satisfies $\nu\circ I= I\circ  R$, with $R$ as in (\ref{act1}). Moreover, $I$
induces homeomorphisms
$$U^{-1}(0)/{S^1}\to \chi(S^2,6),
\text{ and  }
U^{-1}(0)/O(2)\to \chi(F_2).
$$
\end{lemma}
\begin{proof}
Recall that $\pi_1(S^2\setminus 6)=\langle x_1,\dots ,x_6\mid x_6=\overline{x_1x_2x_3x_4x_5}\rangle$.    It follows that $I$ maps
$U^{-1}(0)$ into the subset $\chi(S^2,6)$ of $ \chi(\langle x_1, \dots, x_6\rangle)$.
Any $\rho\in \chi(S^2,6)$ may be conjugated so that $\rho(x_5)=\bbi$, and hence $I$ surjects to
$\chi(S^2,6)$. The stabilizer of $\bbi$ is
$S^1\subset SU(2)$ and therefore   $I$ induces a homeomorphism
$U^{-1}(0)/{S^1}\to \chi(S^2,6).
$
Moreover, 
\begin{eqnarray*} 
\lefteqn{I\left(  R\cdot (X_1,X_2,X_3,X_4)\right) }\\
&=&\left[x_1\mapsto -\bbk X_1\overline{\bbk}, ~x_2\mapsto-\bbk X_2\overline{\bbk}, ~x_3\mapsto-\bbk X_3\overline{\bbk},
~x_4\mapsto-\bbk X_4\overline{\bbk}, ~x_5\mapsto\bbi, ~x_6\mapsto -\bbk \overline{X_1 X_2 X_3 X_4 \bbi}  ~ \overline{\bbk}  ~\right]  
\\
&=& \left[
 x_1\mapsto -  X_1 , ~x_2\mapsto-  X_2 , ~x_3\mapsto -X_3 ,~
x_4\mapsto - X_4 , ~x_5\mapsto-\bbi, \sim x_6 \mapsto -\overline{X_1 X_2 X_3 X_4 \bbi} ~\right] \\
&=& \nu ( I(X_1,X_2,X_3,X_4)).
\end{eqnarray*} 
Hence $ I$ induces a homeomorphism
$$U^{-1}(0)/O(2)=(U^{-1}(0)/S^1)/\langle R\rangle \overset{I}{\underset{\cong}{\to}} \chi(S^2,6)/\langle \nu\rangle\overset{p^*}{\underset{\cong}{\to}} \chi(F_2).
$$
\end{proof}
We prove the local claim in Theorem \ref{thm2.2}  for $s=I(\bbi,\bbi,\bbi,\bbi)$; the cases  $I(\pm\bbi,\pm \bbi,\pm\bbi,\pm\bbi)\in\chi^\ab(S^2,6)$ proceed similarly.   
 Denote by $\mathbb{B}\subset \CC^2$ the ball of radius $\tfrac 1 2$ and  $\mathbb{B}(\ep)$ of radius $\ep>0$ centered at the origin.  These are $O(2)$ invariant with respect to the $O(2)$ action (\ref{act2}).
  The map
 $$
e\colon  \mathbb{B}\to (S^2)^4, ~e(z_1,z_2,z_3,z_4)
=\left(\tfrac{\bbi+z_1\bbj }{\|\bbi+z_1\bbj\|},\tfrac{\bbi+z_2\bbj }{\|\bbi+z_2\bbj\|},\tfrac{\bbi+z_3\bbj }{\|\bbi+z_3\bbj\|},\tfrac{\bbi+z_3\bbj }{\|\bbi+z_3\bbj\|}\right)
 $$
 embeds $\mathbb{B}$,  $O(2)$-equivariantly,  onto    a neighborhood of $ (\bbi, \bbi, \bbi, \bbi)\in (S^2)^4$, equipped with the action (\ref{act1}).

Let $S\colon \mathbb{B}\to \RR$ denote the $O(2)$-equivariant map: $$S(z_1,z_2,z_3,z_4)=\left(\left\| (\bbi+z_1\bbj) (\bbi+z_2\bbj) (\bbi+z_3\bbj) (\bbi+z_4\bbj)\bbi\right\|\right)
 ( U\circ e)(z_1,z_2,z_3,z_4).$$  By construction, $e$  equivariantly embeds $S^{-1}(0)$ with image a neighborhood of $(\bbi,\bbi,\bbi,\bbi)$ in $U^{-1}(0)$, since
 the scaling factor in the definition of $S$ is non-vanishing. 
 
  It follows from Lemma \ref{Dinv} that 
 $e\circ I$ induces  an $R$-equivariant embedding $S^{-1}(0)/S^1\hookrightarrow\chi(S^2,6)$ 
onto a neighborhood of $s=I(\bbi,\bbi,\bbi,\bbi)$, and hence   also an embedding $S^{-1}(0)/O(2)\hookrightarrow \chi(F_2)$
onto a neighborhood of $p^*(s).$   Thus we turn our attention to identifying $(S^{-1}(0),O(2))$ up to equivariant homeomorphism. 
 
  It is straightforward to simplify 
 \begin{align*}S(z_1,z_2,z_3,z_4)&=\Real((\bbi+z_1\bbj) (\bbi+z_2\bbj) (\bbi+z_3\bbj) (\bbi+z_4\bbj)\bbi)\\
 &=  H(z)  +
  \Real(\bbi z_1\overline{z_2} z_3\overline{z_4})
 ,\end{align*}
 where  
 $$
H(z)=\Real(z^* Q z) \text{ with } Q= \frac{\bbi}{2}
\footnotesize
\setlength{\arraycolsep}{2.5pt}
\medmuskip = 1mu 
 \begin{pmatrix} 
0&-1 & 1&-1 \\ 
1&0 &-1&1 \\
-1& 1& 0& -1\\
1&-1 & 1&0
\end{pmatrix}.
$$

\medskip

Define the  $O(2)$-equivariant quartic polynomial  map $\mathbb{S}\colon \CC^4\times[0,1]\to \RR$ by
$$\mathbb{S}(z,t)=H(z)  +
t \Real(\bbi z_1\overline{z_2} z_3\overline{z_4}).$$  
Since $\det Q\ne 0$,   $S_t:=\mathbb{S}(-,t)$ has  an isolated Morse singularity with Hessian equal to $H(z)$ at the origin for every $t\in [0,1]$.
When $t=1, S_1=S$.  When $t=0$, $S_0=H={\rm Hess}(S_t)_0$ and so $S_0^{-1}(0)=H^{-1}(0)$ is a cone on $H^{-1}(0)\cap \partial \mathbb{B}$.  Choose $\ep_0>0$ so that for any $t\in I$, $S_t$ has no   critical points other than $0$ on $\BB(\ep_0)$.

 Let $\pi\colon \mathbb{B}\times [0,1]\to [0,1]$ denote the projection, and   $\BB^*(\ep)=\BB(\ep)\setminus\{0\}$ the punctured ball.    
The signature of $H$ is zero since $S(z,t)$ vanishes on $\RR^4\subset \CC^4$. The {\em invariant }Morse-Bott lemma \cite[Lemma 4.1]{wasserman}, see also \cite{Palais, BH},  implies that  there exists an $0<\ep\leq \ep_0$, a   decomposition $\CC^4= V_+\oplus V_-$ into 4-dimensional, real, $S^1$-invariant orthogonal subspaces,
and a smooth $S^1$-equivariant embedding $\phi\colon \BB(\ep) \times [0,1]\to \BB(\ep_0)\times [0,1]$ of the form $\phi(x,t)=(\phi_1(x,t), t)$, satisfying 
$$\SSS\circ\phi(x,t)=\|x_+\|^2 - \|x_-\|^2 \text{ and }  \phi(0,t)=(0,t),
$$
where $x_\pm\in V_\pm\cap \BB(\ep)$, $x=x_+ +x_-$.

Let $\tfrac{\del}{\del t}$ denote the constant vector field on $\BB(\ep)\times [0,1]$ in the interval direction.  Let $\phi_*(\tfrac{\del}{\del t})$ denote the push-forward to the image of the embedding $\phi$.  
 The smooth vector field on $\phi(\BB(\ep) \times [0,1])$ defined by 
$$X(z,t) =\tfrac{1}{2} \left(\phi_* (\tfrac{\del}{\del t}) + R_* \phi_* (\tfrac{\del}{\del t}) \right) $$
  is $O(2)$-invariant and satisfies
$$d\SSS(X)=\tfrac{1}{2}\left( \tfrac{\del (\SSS\circ \phi)}{\del t} + \tfrac{\del (\SSS \circ R \circ \phi)}{\del t} \right) =\tfrac{1}{2}\left( \tfrac{\del (\SSS\circ \phi)}{\del t} + \tfrac{\del (-\SSS  \circ \phi)}{\del t} \right)=0  , d\pi( X)=\tfrac{\del}{\del t}, \text{ and }  X(0,t)=\tfrac{\del}{\del t}.$$

The flow of $X$ preserves the level sets of $\SSS$ and the critical set $\{0\}\times [0,1]$ is a flow line.  Choose $0<\epsilon_1 \leq \epsilon_0$ so that $\BB(\epsilon_1)\times [0,1] \subset \phi(\BB(\epsilon) \times [0,1])$ and so that the time one flow of $X$ defines a $O(2)$-equivariant smooth embedding $E\colon \BB(\epsilon_1)\times \{0\} \hookrightarrow \BB(\epsilon_0)\times \{1\}$.  The restriction $E\colon S_0 ^{-1}(0)\cap \BB(\epsilon_1) \to S_1 ^{-1}(0) \cap \BB(\epsilon_0)$ is an $O(2)$-equivariant homeomorphism onto its image.

 It follows that  $0$ has a neighborhood in $S^{-1}(0)\subset \BB$  that is  $O(2)$-equivariantly 
 homeomorphic to the cone $H^{-1}(0)\subset \CC^4$, where $O(2)$ acts as in (\ref{act2}). 
 Lemma \ref{Dinv}  then implies that a neighborhood of $s$ in $\chi(S^2,6)$ is homeomorphic
 to $H^{-1}(0)/S^1$ and a neighborhood of $p^*(s)$ in $\chi(F_2)$ is homeomorphic
 to $H^{-1}(0)/O(2)$, so we now focus our attention on $O(2)$-equivariantly identifying the   zero set of the quadratic form $H$.

\medskip

 Recall $H(z)=\Real(z^*Qz)$. 
 The matrix $Q$ is Hermitian, has non-zero determinant, and satisfies $\overline{Q}=-Q$. 
  This implies that $\CC^4$ has an  orthogonal eigenbasis of the form $\{v_1, v_2, \overline{v_1}, \overline{v_2}\}$ satisfying 
 $Q v_i=\lambda_i v_i$ and $Q\overline{v_i}=-\lambda_i\overline{v_i}$ for some $\lambda_i>0,~i=1,2$. 
 After rescaling the basis elements, we can assume  $H(v_i)= 1$ and  $H(\overline{v_i})=-1$.  It follows that for any complex linear combination, \begin{align*}
 H(z_1v_1+z_2v_2+w_1\overline{v_1}+w_2\overline{v_2}) =
 |z_1|^2 +   |z_2|^2
 - |w_1|^2 
 - |w_2|^2      
 =0.
 \end{align*}
The (left) complex linear isomorphism $\psi\colon \HH\times \HH\to \CC^4$ given by
$$ \psi(z_1+z_2\bbj, w_1+w_2\bbj)=z_1v_1+z_2v_2+w_1\overline{v_1}+w_2\overline{v_2}
$$
restricts to a homeomorphism $\psi\colon {\rm Cone}(S^3\times S^3) \to H^{-1}(0)$ satisfying
$$\psi\left(e^{\theta\bbi}(z_1+z_2\bbj, w_1+w_2\bbj)\right)=e^{\theta\bbi}\psi(z_1+z_2\bbj, w_1+w_2\bbj)$$  
and\begin{eqnarray*} 
\overline{\psi(z_1+z_2\bbj, w_1+w_2\bbj)}&=& 
  \psi(\overline{w_1}+\overline{w_2}\bbj, \overline{z_1}+\overline{z_2}~\bbj)\\
&=& \psi( \bbj(w_1+w_2\bbj)\bar\bbj,\bbj  (z_1+z_2\bbj)\bar\bbj)\\  &=&\psi(R(z_1+z_2\bbj, w_1+w_2\bbj)), \end{eqnarray*}
with $R$ as in the $O(2)$ action  (\ref{O2HH}).
In other words, $\psi:{\rm Cone}(S^3\times S^3)\to H^{-1}(0)$ is an $O(2)$-invariant homeomophism, where $O(2)$ acts on $S^3\times S^3$ via  (\ref{O2HH}).

\begin{lemma}\label{kicka}
 Let $O(2)$ act on $S^3\times S^3$ as in (\ref{O2HH}), and let $(S^3\times_{S^1} S^3, R)$ denote the induced $\ZZ/2$ action on the $S^1$ quotient. Then $(S^3\times_{S^1} S^3, R)$ is equivariantly  diffeomorphic  to $(UTS^3, \tau)$, where $\tau$ denotes the differential of reflection of $S^3$ through the hyperplane ${\rm Span}\{1,\bbi,\bbk\}$.  \end{lemma}
\begin{proof}

  The unit tangent bundle of $S^3$ is expressed in quaternionic notation as $$UTS^3=
  \left\{(C,D)\in S^3\times S^3\mid  \Real\left(\overline{C} D\right)=0\right\}.$$  
  The differential of the reflection  of $S^3$  through the hyperplane ${\rm Span}\{1,\bbi,\bbk\}$ is given by $(C,D)\mapsto \left(\bbj \overline{C} ~\overline{\bbj}, \bbj \overline{D} ~\overline{\bbj}\right)$. 
  
  Consider the smooth, $S^1$-invariant map $$ M\colon S^3 \times S^3\to  UTS^3 \subset S^3\times S^3, ~ M(A,B)=\left(\overline{A} B, \overline{A} \bbi B\right). $$  Since  
$$
M( R\cdot (A,B))=\tau\circ M(A,B),
$$ 
 $M$ descends to a involution-preserving
 map  $(S^3 \times_{S^1} S^3, R)\to  (UTS^3,\tau)$. We now show that  this map is a diffeomorphism, by showing that  
 $M\colon S^3\times S^3\to UTS^3$ is a principal $S^1$ bundle, completing the proof of the lemma. 
 
Consider the diffeomorphisms $L:S^3\times S^3\to S^3\times S^3$ defined by $L(A,B)=(A\overline B,B)$ and   $P:UTS^3\to S^3\times S^2$ defined by 
 $P(C,D)=(C,\overline D C)$.  Let $\eta:S^3\to S^2$ denote the Hopf principal $S^1$ fibration,  given by $\eta(B)=\overline B\bbi B$. 
 The diagram
 \[
 \begin{tikzcd}
 S^3\times S^3\arrow[r, "M"]\arrow[d, "L","\cong"']&UTS^3\arrow[d, "P","\cong"']\\
 S^3\times S^3\arrow[r, "1\times \eta"]&S^3\times S^2
 \end{tikzcd}
 \]
commutes. The fibers of $(1\times \eta)\circ L$ are the $S^1$ orbits of (\ref{O2HH}), and therefore $M$ is a principal $S^1$ bundle. \end{proof}

Corollary \ref{ascor}, Theorem \ref{essfive},  and Lemma \ref{kicka} complete the proof of Theorem \ref{thm2.2}.\end{proof}

\section{Decomposing a genus two surface along a separating curve}\label{basics2}

 This section considers a decomposition of a genus two surface into punctured tori, and it establishes a series of local differential-topological properties of various restriction maps between  character varieties associated to this decomposition; these properties are  used below to prove Theorem \ref{thmA}.

\subsection{The character of the commutator} The commutator $[r_-, s_-]$ is represented by an embedded separating curve $C$ (see Figure \ref{fig0fig}) which decomposes
$F_2$ into two punctured tori labelled $F_\pm$:
  \begin{equation}\label{decomp2}
 F_2=F_-\cup_C F_+ .
 \end{equation}
 The  character of the commutator $[r_-,s_-]$ is, by definition,  the  map
 \begin{equation}\label{kappa1}
\kappa\colon \chi(F_2)\to [-1,1], ~ \kappa(\rho)=\Real(\rho([r_-,s_-])).
\end{equation}
 The restriction of $\kappa$ to $\chi(C)$ is a homeomorphism to $[-1,1]$, and the restriction to 
 $\chi^\ab(C)$  a diffeomorphism to the open interval $(-1,1)$ (Section \ref{listofch}).

\subsubsection{Notation} The character $\kappa$ defines a function from each of the  character varieties  $\chi(F_2), \chi(F_\pm)$, and $\chi(C)$ to the interval $[-1,1]$.  We denote the pullback of the two maps $\kappa_\pm \colon  \chi (F_\pm) \to \chi(C)$ by $\chi(F_-)\times_\kappa\chi(F_+)$.
 For any subset $J\subset [-1,1]$,  we use the subscript ``$J$'' to denote
the preimage of $J$ under the appropriate $\kappa$ or $k$. For example, 
$$\chi(F_2)_J=\kappa^{-1}(J)=\chi(F_2)\cap \{\Real( \rho([r_-,s_-]))\in J\},$$
and $$\chi(F_-)\times_\kappa\chi(F_+)_J=\{(\rho_-,\rho_+)\in \chi(F_-)\times \chi(F_+)\mid  \kappa(\rho_-)=\kappa(\rho_+)\in J\}$$
 When $J=\{s\}$ is a singleton,  we write $\chi(F_2)_s$ rather than $\chi(F_2)_{\{s\}}$, and similarly for  $\chi(F_\pm)$ and $\chi(C)$.

\subsection{The character variety pullback diagram  associated to the decomposition of $F_2$}

 Applying the contravariant $SU(2)$ character variety functor to the diagram  \begin{equation*}\label{SVK1}
 \begin{tikzcd}
F_2&F_+ \arrow[l,hook]\\
  F_- \arrow[u,hook]&C\arrow[u,hook]\arrow[l,hook]
\end{tikzcd} 
\end{equation*}
produces the following   commutative diagram, for any $J\subset [-1,1]$.
\begin{equation}\label{SVK3} \begin{tikzcd}
\chi(F_2)_J\arrow[rd,swap,"\phi"]\arrow[bend right=7,swap, "\phi_-"]{rdd}\arrow[bend left=2,"\phi_+"]{rrd}&&     & \\
&\chi(F_-)\times_{\kappa}\chi(F_-)_J\arrow[r,swap,"\text {Proj}_+"]\arrow[d,"\text{Proj}_-"]&\chi(F_+)_J\arrow[d, "\kappa_+"]  & \\
&\chi( F_-)_J\arrow[r,swap,"\kappa_-"]&\chi(C)_J\arrow[r,"\cong", "\kappa" ']  &J
\end{tikzcd} 
\end{equation}
The fibers of the canonical map $\phi$ to the pullback:
\begin{equation}\label{phi2}
\phi=\phi_-\times \phi_+\colon \chi(F_2)\to \chi(F_-)\times_{\kappa}\chi(F_+),
\end{equation} 
 known as {\em gluing} or {\em bending  parameters} for the decomposition, are identified in Proposition \ref{fibers2}.

We list some   properties of the  maps in (\ref{SVK3})  in the following lemma.

 \begin{lemma}\label{WeilorWeyl}  \hfill
 
 \begin{enumerate}
\item Taking $J=[-1,1]$, every map in (\ref{SVK3}) is surjective.

 \item Taking $J=[-1,1)$,   $\chi(F_2)_J$  is a smooth open 6-dimensional submanifold of $\chi^*(F_2)$. Also $\chi(F_\pm)_J=\chi^*(F_\pm)$; these are smooth open 3-dimensional  manifolds.
The restrictions of $\phi_+$ and $\phi_-$ to these submanifolds are submersions.
 
 \item Taking $J=(-1,1)$,  every space in (\ref{SVK3}) is a smooth manifold  and  every map in the diagram  is a smooth submersion.  Moreover, $\dim \chi(F_-)\times_\kappa\chi(F_+)_{(-1,1)}=5.$

\end{enumerate}
 
\end{lemma}
 \begin{proof}

    Any $\rho_-\colon \pi_1(F_-)\to SU(2)$ extends to $\rho\colon \pi_1(F_2)\to SU(2)$  by setting $\rho(r_+)=\rho_-(s_-)$ and $\rho(s_+)=\rho_-(R_-)$. Hence $\phi_-$ is surjective, and similarly $\phi_+$, is surjective.
 The Seifert-Van Kampen theorem  implies that $\phi$ is surjective, and commutativity implies   Proj$_\pm$ are surjective. The assignment $$r_-\mapsto \bbi,~ s_-\mapsto t\bbi +\sqrt{1-t^2}\bbj,~ t\in[0,1]$$
satisfies $\Real([r_-,s_-])=2t^2-1$, showing that $\kappa_-$ is surjective.
 The same argument shows  $\kappa_+$ is surjective, completing the proof of the first assertion.
  
 \medskip
 
 The remaining assertions depend on the following  easy calculations.
 Suppose  that $\rho\in \chi(F_2)$ satisfies $\kappa(\rho)\ne 1$. Then $\rho$ and its restrictions $\rho_\pm$ to $\pi_1(F_\pm)$ are non-abelian. Hence $H_0(F_2;su(2)_{ad\rho})=0$ and 
 $H_0(F_\pm; su(2)_{ad\rho})=0$.  Applying Poincar\'e duality we conclude that $H^2(F_2; su(2)_{ad\rho})=0$  and $H^2(F_\pm, C; su(2)_{ad\rho})=0$.
Excision shows $ H^2(F_2, F_\mp;su(2)_{ad\rho})\cong H^2(F_\pm, C; su(2)_{ad\rho})=0.$ If in addition $\kappa(\rho)\ne \pm1$, then 
 $H^1(C;su(2)_{ad\rho})\cong   \RR $.
 \medskip

We prove the second assertion. First $\chi(F_2)_{[-1,1)}$ is an open subset of $\chi(F_2)$ contained in the smooth locus $\chi^*(F_2)$, and hence by Proposition \ref{dimension} a smooth 6-manifold.
Similarly,  $\chi(F_\pm)_{[-1,1)}$ is an open subset of $\chi(F_\pm)$ contained in the smooth locus $\chi^*(F_\pm)$, an open 3-manifold,  as explained in Section \ref{listofch}. In fact 
$\chi(F_\pm)_{[-1,1)}=\chi^*(F_\pm)$, since $h\colon \langle r,s\rangle\to SU(2)$ is abelian if and only $\Real(h([r,s]))=1$.

Using Weil's theorem, surjectivity of  $d\phi_\pm$ at $\rho\in \chi(F_2)_{[-1,1)}$ is equivalent to  $H^1(F_2;su(2)_{ad \rho})\to H^1(F_\pm;su(2)_{ad \rho})$ being surjective.    
The exact sequence of the pair $(F_2,F_\pm)$
\begin{align*}
\cdots\to H^1(F_2;su(2)_{ad \rho})\xrightarrow{(d\phi_\pm)_\rho}H^1(F_\pm ;su(2)_{ad \rho})\to H^2(F_2,F_\pm;su(2)_{ad \rho})\to\cdots , \end{align*}
 and the fact noted above that $H^2(F_2,F_\pm;su(2)_{ad\rho})=0$ if $\kappa(\rho)\neq 1$  imply
that $d\phi_\pm$ are submersions over $J=[-1,1)$, completing the proof of the second assertion.  

\medskip
 From the second assertion we know that $\phi_{\pm}\colon \chi(F_2)_{(-1,1)}\to \chi(F_\pm)_{(-1,1)}$ are submersions of smooth manifolds.  Since $\kappa\colon \chi(C)_{(-1,1)} \to (-1,1)$ is a diffeomorphism, it remains only to  show that
$\kappa_\pm$ and $\phi$ are submersions over $(-1,1)$ and that $\chi(F_-)\times_\kappa\chi(F_+)_{(-1,1)}$ is a smooth 5-manifold.

Using Weil's theorem, surjectivity of  $d\kappa_\pm$ at $\rho_\pm\in \chi(F_\pm)_{(-1,1)}$ is equivalent to  $H^1(F_\pm;su(2)_{ad \rho})\to H^1(C;su(2)_{ad \rho})$ being surjective.  
This follows  from the exact sequence of the pair $(F_\pm,C)$
 and the fact that $H^2( F_\pm, C;su(2)_{ad\rho} )=0$.

\medskip

The calculations $H^2(F_2;su(2)_{ad\rho})=0$ and $H^1(C;su(2)_{ad\rho})\cong  \RR$ when $|\kappa|<1$, combined with 
the Mayer-Vietoris sequence,
$$\cdots \to H^1(F_2;su(2)_{ad\rho})\xrightarrow{d\phi_-\oplus d\phi_+} H^1(F_-;su(2)_{ad\rho})\oplus H^1(F_+;su(2)_{ad\rho})\xrightarrow{\alpha} H^1(C;su(2)_{ad\rho})\to 0.
$$
shows that $\chi(F_-)\times_\kappa\chi(F_+)_{(-1,1)}$ is  a smooth 5-manifold, cut out transversely  as the zero set of the map 
 $$ \chi^*(F_-)\times\chi^*(F_+)_{(-1,1)} \to \RR, (\rho_-,\rho_+)\mapsto \kappa(\rho_+)-\kappa(\rho_-).$$
  This Mayer-Vietoris sequence also
 identifies the tangent space of $\chi(F_-)\times_{\kappa}\chi(F_-)_{(-1,1)}$ with the image of the differential of $\phi$ at $\rho$, proving that $\phi$ is a submersion over $(-1,1)$. 
 \end{proof}

Looking ahead, Proposition \ref{prop1} below, due to Goldman \cite{GoldmanErgodic}, shows that $\chi(F_\pm)$ is homeomorphic to a closed topological 3-ball in $\RR^3$, with smooth interior  $\chi^*(F_\pm)$.   Lemma \ref{JAB} shows that $\chi(F_-)\times_\kappa\chi(F_+)$
is homeomorphic to a cone on $S^2\times S^2$, with
 $\chi(F_-)\times_\kappa\chi(F_+)_{(-1,1)}$   diffeomorphic to $S^2\times S^2\times(-1,1)$.

\section{ $\chi(F_2)$ is homeomorphic to $\CC P^3$}\label{princ}
 We now turn to the identification of the compact   6-manifold $\chi(F_2)$ with $\CC P^3$.  This is done by comparing the character $\kappa$ to a certain Morse-Bott function on $\CC P^3$.

\subsection{A  decomposition of $\CC P^3$}
Consider the function
$$\kappa_{\CC P^3}\colon \CC P^3\to [-1,1]$$
\begin{equation}\label{Coltrane} \kappa_{\CC P^3}([x:y:z:w])=1-2\frac{|x^2+y^2+z^2+w^2|^2}{\left(|x|^2+|y|^2+|z|^2+|w|^2\right)^2}.
\end{equation}
Then $\kappa_{\CC P^3}$ has minimum $-1$ on $\RR P^3$,  and maximum $1$ on the smooth quadratic complex surface $$Q=\left\{x^2+y^2+z^2+w^2=0\right\}.$$

It can be verified by hand  that $\kappa_{\CC P^3}$ is Morse-Bott with only the maximum and minimum as critical values. Hence $\kappa_{\CC P^3}$ separates $\CC P^3$ into tubular neighborhoods
of $\RR P^3$ and $Q$.
 Alternatively, 
Uchida \cite[Ex. 3.2]{Uchida} gives the following  construction.  Let $SO(4)$ act on $\CC P^3$ via $SO(4)\subset GL(4,\CC)$. 
Then, \cite[Lemma 1.2.1]{Uchida} asserts
that $\CC P^3/SO(4)$ is homeomorphic to the unit interval $ [-1,1]$, the endpoints correspond to the two non-principal orbits $\RR P^3$ and $Q$, and  the preimages of $[-1,0]$ and $ [0,1]$ are (closed) tubular neighborhoods of $\RR P^3$ and $Q$.   It is easily checked that the orbits of this $SO(4)$ action coincide with the level sets of $\kappa_{\CC P^3}$.

\medskip

As a complex submanifold of $\CC P^3$,  $Q$ is oriented, as is its normal bundle. 
There exists an orientation-preserving  diffeomorphism  (the Segre embedding) $S^2\times S^2\to Q\subset \CC P^3$ such that the Euler class $e$ of the normal bundle of $Q$ satisfies $e( S^2\times \{p\} )=2=e( \{p\}\times S^2  )$. The tubular neighborhood theorem provides   an  orientation-preserving embedding  of   the oriented Euler class $(2,2)$ $D^2$ bundle   $E\to S^2\times S^2$  into $\CC P^3$,
with image the closed tubular neighborhood $\kappa_{\CC P^3}^{-1}([0,1])$ of $Q$.
\[\begin{tikzcd}
E\arrow[d]\arrow[r, "\cong"]& \kappa_{\CC P^3}^{-1}([0,1])\arrow[d]\arrow[dr, hook]&\\
S^2\times S^2\arrow[r,"\cong"]&Q\arrow[r,hook]&\CC P^3
\end{tikzcd}
\]
Going forward we identify  $E$ with $\kappa_{\CC P^3}^{-1}([0,1])\subset \CC P^3$, and $S^2\times S^2$ with $Q$.  

\medskip

Let $s_1,s_2\in H^2(S^2\times S^2)$ denote the classes satisfying 
$$s_1(S^2\times \{p\})=1,~s_1( \{p\}\times S^2)=0,~s_2(S^2\times \{p\})=0,~s_2( \{p\}\times S^2)=1.$$
Then $\{s_1,s_2\}$ generates $H^2(S^2\times S^2) $, $s_1s_2$ generates $H^4(S^2\times S^2)$, $s_1s_2\cap [S^2\times S^2]=1$, and  $e=2s_1+2s_2$.  The bundle projection induces an isomorphism $H^*(S^2\times S^2)\to H^2(E)$; we use the notation $s_1,s_2, s_1s_2$ also for the image of these classes by this isomorphism.

\medskip

Thus $\kappa_{\CC P^3}$ determines  a decomposition of  $\CC P^3$:
\begin{equation}\label{CP3eas}\CC P^3=\{\kappa_{\CC P^3}\leq 0\}\cup_{\{\kappa_{\CC P^3}= 0\}}\{\kappa_{\CC P^3}\ge 0\}=
\left(\RR P^3\times D^3\right) \cup_{\RR P^3\times S^2 }E.\end{equation}

\begin{proposition}\label{P=NPforreal} Let  $d\colon \RR P^3\times S^2\to \partial E$ be any diffeomorphism and
let $X(d)$ denote the smooth compact 6-manifold obtained by gluing $\RR P^3\times D^3$ and $E$ 
 using $d$. Then $X(d)$ is diffeomorphic 
 to $\CC P^3$. 

\end{proposition}
\noindent{\em Remark.} It is unknown to the authors whether every diffeomorphism of $\RR P^3\times S^2$ extends over  one of $\RR P^3\times D^3$ or $E$. If true, then Proposition \ref{P=NPforreal}  follows as a trivial consequence. 

\begin{proof}  
  As we have shown above, for at least one choice of $d$, call it $d_0$, $X(d_0)=\CC P^3$.    Recall that as a ring, $H^{*}(\CC P^3)= \ZZ[h]/\left\langle h^4\right\rangle, ~h\in H^2(\CC P^3),$ and $h^3\in H^6(\CC P^3)$ is the orientation class.
The calculations $w_2(\CC P^3)=0$ (\cite[page 133]{milnor-stasheff}) and $p_1(\CC P^3)=4h^2$ (\cite[page 178]{milnor-stasheff}) are standard. 

\medskip

 Fix a diffeomorphism $d\colon \RR P^3\times S^2\to \partial E$.  Let $i_d\colon E\hookrightarrow X(d)=(\RR P^3\times S^2)\cup_d E$ denote the inclusion (we use $i_d$ also to denote the induced map on cohomology).
  Orient $X(d)$ so that $i_d$ is orientation preserving.  
  The Seifert-Van-Kampen theorem implies that $X(d)$ is simply connected for any $d$.
 A straightforward calculation,  using the long exact sequence of the pair  $(X(d), E)$, excision, Poincar\'e duality and the universal coefficient theorem,  proves that
 \begin{enumerate}
\item $H^i(X(d))=\ZZ $ for $i=0,2,4,6$, and equals zero otherwise,
\item $
 0\to H^4(X(d))\cong\ZZ   \xrightarrow{i_d}  H^4(E)\cong\ZZ \to H^5(X(d),E)\cong\ZZ/2\to 0
 $ is exact, and 
 \item $
 0\to H^2(X(d))\cong\ZZ \xrightarrow{i_d} H^2(E)\cong\ZZ\oplus\ZZ \to H^3(X(d),E)\cong\ZZ\to 0
 $ is exact.
\end{enumerate}

Exact sequence (2) shows  that the image of $H^4(X(d))  \xrightarrow{i_d}  H^4(E)\cong H^4(S^2\times S^2)\cong \ZZ\langle s_1s_2\rangle$
is the unique index two subgroup. Hence, for each $d$, there exists a unique  generator $k_d\in H^4(X(d))$ satisfying
$$i_d(k_d)=2s_1s_2.$$
Poincar\'e duality  implies that there is a unique generator   $h_d\in H^2(X(d))$ satisfying
$h_dk_d\cap [X(d)]=1$. 
Define $n_d\in \ZZ $ by $$h_d^2=n_d k_d,$$
In the case of $\CC P^3$, $h_{d_0}=h$,  $k_{d_0}=h_{d_0}^2$, and $n_{d_0}=1$.

\medskip

According to \cite[Theorem 5]{Wall-classificationV}, $X(d)$ is diffeomorphic to $\CC P^3$ if  
$$w_2(X(d))=0,~
 p_1(X(d))=4 k_d, \text{ and }
n_d=1.$$
A calculation with the long exact sequence of the pair $(X(d), \RR P^3\times D^3)$ with $\ZZ/2$ coefficients establishes that   the restriction $H^2(X(d);\ZZ/2)\to  H^2(\RR P^3\times D^3;\ZZ/2)=\ZZ/2$ is an isomorphism.
Since $\RR P^3\times D^3$ has a trivial tangent bundle, naturality implies that $w_2(X(d))$ maps to zero in $H^2(\RR P^3\times D^3;\ZZ/2)$  and therefore $w_2(X(d))=0$. This verifies the first condition.
Since $p_1(\CC P^3)=4h_{d_0}^2$,   naturality shows  that the manifold $E$ has $p_1(E)=8s_1s_2$. This  implies that $p_1(X(d))=4k_d$, verifying the second condition. 
 It remains, then, to show that $n_d=1$.

\begin{lemma}\label{abcconjecture} The image of $i_d\colon H^2(X(d))  \to H^2(E)$ is independent of $d$.
\end{lemma} 
\begin{proof} 
The ladder from the exact sequence (3)  to the cohomology sequence for the pair $ (\RR P^3\times D^3, \RR P^3\times S^2)$:
\[
\begin{tikzcd}
0\arrow[r]&H^2(X(d))\arrow[r,"i_d"]\arrow[d]& H^2(E)\arrow[r] \arrow[d,"d^*\circ r"]&H^3(X(d),E)\arrow[r]\arrow[d,"\cong"]&0\\
0\arrow[r]&H^2(\RR P^3\times D^3)\arrow[r] & H^2( \RR P^3\times S^2)\arrow[r] &H^3(\RR P^3\times D^3, \RR P^3\times S^2)\arrow[r] &0
\end{tikzcd}
\]
commutes, where $r\colon H^2(E)\to H^2(\partial E)$ is the restriction, and $d^*\colon H^2(\partial E)\to H^2( \RR P^3\times S^2)$ is induced by the diffeomorphism $d$.   The bottom exact sequence is computed via the Kunneth theorem to be $0\to \ZZ/2\to \ZZ/2\oplus \ZZ\to \ZZ\to 0$ and therefore   $d^*\circ r\circ i_d(h_d)$ lies in  the torsion subgroup $T\cong\ZZ/2\subset H^2( \RR P^3\times S^2)$. Since the isomorphism $d^*$ preserves torsion subgroups, 
$${\rm image}(i_d) = r ^{-1}({\rm Torsion}(H^2(\partial E))),$$ proving the lemma.
 \end{proof}

 We can now complete the proof of Proposition \ref{P=NPforreal}.
 Write $i_{d_0}(h_{d_0})=as_1+bs_2$.  Lemma \ref{abcconjecture} implies that
  $$i_{d}(h_d)^2= (\pm i_{d_0}(h_{d_0}))^2= (as_1+bs_2)^2=2abs_1s_2.$$ 
  On the other hand, naturality of cup products gives
  $$i_{d}(h_d)^2=i_{d}(h_d^2)=i_{d}(n_d k_d)=2n_d s_1s_2,$$
and therefore $ab=n_d$.  Since $n_{d_0}=1$, $1=n_d$.  \end{proof}

\subsection{A decomposition of $\chi(F_2)$ induced by $\kappa$.}
 Lemma \ref{WeilorWeyl} implies  that     $\kappa\colon  \chi(F_2)_{(-1,1)}\to (-1,1)$ is a smooth submersion.
Therefore,
the character  $\kappa$ determines the decomposition of the  manifold $\chi(F_2)$ into two compact  topological manifolds along their common boundary: 
\begin{equation}\label{SVK1prime}
\chi(F_2)=\chi(F_2)_{[-1,0]}\cup_{\chi(F_2)_0} \chi(F_2)_{[0,1]}.
\end{equation}
 The  manifold $\chi(F_2)_{[-1,0]}$ is a smooth, codimension zero  submanifold of   the smooth manifold $\chi^*(F_2)$, with smooth boundary $\chi(F_2)_0$. 
 
 We now state two theorems about $\chi(F_2)_{[-1,0]} $ and $\chi(F_2)_{[0,1]}$ and show how these two theorems imply Theorem B (which we state as a corollary below).   
 
\begin{theorem}\label{thm7} The space $\chi(F_2)_{-1}$ is diffeomorphic to $\RR P^3$ and the inclusion  $\chi(F_2)_{-1}\subset \chi^*(F_2)$ is a smooth  embedding.  The smooth compact manifold with boundary 
$\chi(F_2)_{[-1,0]}$ is diffeomorphic to $\RR P^3\times D^3$, and is a closed tubular neighborhood of $\chi(F_2)_{-1}$ in $\chi^*(F_2)$. 
\end{theorem}
The proof of Theorem  \ref{thm7} is given  in Section \ref{POT8} below, and a quick symplectic proof
is given in Section \ref{w2}. 
 More challenging   is the  proof of the following theorem, given  in Section \ref{7.6}.

\begin{theorem}\label{thm8} There exists a homeomorphism of compact manifolds with boundary,
$$h\colon \chi(F_2)_{[0,1]}\to E,$$  whose restriction to   the collar $\chi(F_2)_{[0,0.5]} \subset \chi^*(F_2)$ is a  diffeomorphism onto its image.
\end{theorem}

\begin{corollary}\label{cor10}
 $\chi(F_2)$ is homeomorphic to $\CC P^3$.  
\end{corollary}
\begin{proof}
The homeomorphism $h$ of 
 Theorem \ref{thm8}  can be used to pull back the smooth structure on $E$ to a smooth structure on  $\chi(F_2)_{[0,1]}$ which agrees, on $\chi(F_2)_{[0, 0.5]}$, with the smooth structure inherited from 
  $\chi(F_2)_{[0,0.5]}\subset \chi(F_2)_{[-1,0.5]}\subset\chi^*(F_2)$. This endows $\chi(F_2)=\chi(F_2)_{[-1,0]}\cup \chi(F_2)_{[0,1]}$ with a smooth structure whose restriction to $\chi(F_2)_{[-1,0.5]}\subset \chi^*(F_2)$ coincides with the given (Zariski) smooth structure on $\chi^*(F_2)$.

  Theorem \ref{thm7} then implies that with this smooth structure,  $ \chi(F_2)$ is diffeomorphic to $X(d)$ for some diffeomorphism $d\colon \RR P^3\times S^2\to \partial E$.
  Proposition \ref{P=NPforreal} then implies that equipped with this smooth structure, the topological manifold $\chi(F_2)$ is diffeomorphic to $\CC P^3$.
 \end{proof}

   Theorems \ref{thm7} and \ref{thm8} immediately imply  Theorem \ref{thmA} and Corollary \ref{cor10} is a restatement of Theorem \ref{thmB} in the introduction.

\section{Proofs of Theorems \ref{thm7} and \ref{thm8}}\label{pfthm1}

\subsection{The Goldman pillow}\label{GP}

The fundamental groups of the punctured tori $F_\pm$ are free of rank two, generated by $r_\pm$ and $s_\pm$.  Goldman identified the character variety of the free group of rank two $\langle r,s\rangle$ with   a 3-ball, as we next explain.

The {\em Goldman pillow} \cite{GoldmanErgodic} is the semi-algebraic subspace of the 3-dimensional cube $[-1,1]^3$:
$$B=\left\{(x,y,z)\in \RR^3\mid   x^2+y^2+z^2-2xyz\leq 1, |x|\leq 1, |y|\leq 1,|z|\leq 1\right\}.$$
Its   boundary   is the {\em pillowcase}: $$\partial B=\left\{(x,y,z)\mid   x^2+y^2+z^2-2xyz=1, |x|\leq 1, |y|\leq 1,|z|\leq 1\right\}.$$
The set of four points:
\begin{equation}\label{SSigma} \{(1,1,1),(-1,1,-1),(1,-1,-1),(-1,-1,1)\}\in  \partial B \end{equation}
are called the {\em corners}.  The subset $B\setminus \partial B$ is open in $\RR^3$.

The pillowcase $\partial B$ has an alternative description as the orbit space of $\TT^2$ by the elliptic involution 
$\iota(e^{\alpha\bbi},e^{\beta\bbi})=(e^{-\alpha\bbi},e^{-\beta\bbi})$; the map $\TT^2\to \partial B$ taking $(e^{\alpha\bbi},e^{\beta\bbi})$ to $(\cos(\al),\cos(\beta), \cos(\alpha-\beta))$ is a   double branched cover, branched over the four corners, which have coordinates $\alpha, \beta\in \{0,\pi\}$.

\medskip

The following proposition is  well-known  and a straightforward exercise; see  \cite[Section 2]{GoldmanErgodic}, \cite[Proposition 3.1]{jeffrey-weitsman}, and  \cite[Proposition 3.4]{choi}.
\begin{proposition}\label{prop1} The subset  $B\subset \RR^3$ is homeomorphic to the closed 3-ball.  The map 
\begin{equation}\label{G}W\colon \chi\left(\langle r,s\rangle\right)\to \RR^3, ~ W([\rho])=(\Real(\rho(r)),\Real(\rho(s)),\Real(\rho(r\bar s)).\end{equation}
 is an embedding with image 
 $B.$ The pillowcase $\partial B$ is  the image of the subvariety $\chi^{\rm ab}\left(\langle r,s\rangle\right)$. The four corners  form the image of $\chi^{\rm cen}\left(\langle r,s\rangle\right)$. The center $(0,0,0)$ is the image of the conjugacy class of the non-abelian representation $\rho_0\colon r\mapsto \bbi, s\mapsto \bbj$.  The restriction of $W$ to $\chi^*\left(\langle r,s\rangle\right)$ is a diffeomorphism onto  the open Goldman pillow   $B\setminus \partial B$. 
\end{proposition}

\subsection{A Morse function on the Goldman pillow}
Define 
\begin{equation}\label{defng}
k\colon \RR^3\to \RR, ~ k(x,y,z)=2\left(x^2+y^2+z^2 -2xyz\right)-1.
\end{equation}
Then  $$B=\{k\leq 1\}\cap [-1,1]^3\text{  and  }\partial B=\{k=1\}\cap [-1,1]^3.$$
\begin{proposition} \label{GFLonB}\hfill
\begin{enumerate}
\item 
  $\kappa =k\circ W,$ where $\kappa(\rho)=\Real(\rho([r,s])$.
  \item $k\colon \RR^3\to \RR$ is Morse with five critical points: a local minimum at $(0,0,0)$, and four index one critical points at the corners.
  \item For each of the four corner points $\sigma$, the open  line segment $(0,1)\ni t\mapsto t\sigma$ is  the image of a  gradient flow line of $k$  limiting to  $(0,0,0)$ at $-\infty$ and to   $\sigma$ at $\infty$.   
 \item All other upward gradient flow lines of $k$ from $(0,0,0)$
  intersect $\partial B $, transversely.
\end{enumerate}
 \end{proposition}

\begin{proof}

The trace identities for $SL(2)$,  
 ${\rm Tr}(ab)+{\rm Tr}\left(a^{-1}b\right)={\rm Tr}(a){\rm Tr}(b), ~ {\rm Tr}\left(a^{-1}\right)={\rm Tr}(a)$  imply that, for any $a,b\in SU(2)$, 
$$
\Real(ab)+\Real(\bar a b)=2\Real(a)\Real(b), ~ \Real(\overline{a})=\Real(a).
$$

 Pick   $ \rho \in \chi(\langle r,s\rangle)$ and let $a=\rho(r), b=\rho(s)$.   Then 
 $(x,y,z)=\left(\Real(a), \Real(b),\Real \left(a\overline{b}\right)\right)=W(\rho)$.
\begin{align*}
\kappa(\rho)= \Real \left(\left(a b \overline{a}\right)  \overline{b}\right)  &= -  \Real \left( 
\left(a \overline{b}\right)\left( \overline{a}\overline{b}\right)\right)+2  \Real \left(a  b \overline{a} \right)\Real \left(\overline{b}\right)  \\
  &=\Real \left( b \overline{a} ^2\overline{b}\right)-2\Real \left(a  \overline{b}\right)\Real \left(\overline{a} \overline{b}\right) + 2 \Real(b)^2\\
  &=\Real \left(\overline{a} ^2\right) -2z\left(-\Real \left(a\overline{b}\right)+ 
  2\Real \left(\overline{a}\right)\Real \left(\overline{b}\right)\right) + 2 \Real \left(b \right)^2\\
  &=-1+2x^2+2z^2-4 zxy +2y^2=k(x,y,z)=k\circ W(\rho).
  \end{align*}

 All remaining statements follow easily from the calculation
 $\nabla k=4 (x-yz,y-xz,z-xy)$.
\end{proof}

Since $k$ is Morse with a local minimum at $(0,0,0)$ and no critical values in $(-1,1)$, the level sets $k^{-1}(t)$ are smooth 2-spheres for $-1<t<1$.  Fix a diffeomorphism of $S^2$ with the 0-level set $$\alpha_0\colon S^2\xrightarrow{\cong} k^{-1}(0)\subset B.$$ 

\begin{corollary}\label{GF2} There is a  homeomorphism
$$\alpha\colon S^2\times (-1,1]\cong  B\setminus \{(0,0,0)\}
$$
which satisfies 
 $k(\alpha(\sigma,t))=t$ and $\alpha(\sigma,0)=\alpha_0(\sigma)$,   and such that  $\alpha$
 is a local diffeomorphism  away from the four corner points.
\end{corollary}
\begin{proof}
Let $X=\frac{\nabla k}{|\nabla k|^2}$.   The vector field $X$ is defined and smooth on the complement of the critical points of $k$ in $\RR^3$, an open set which contains  $B\setminus(\{\rm corners\}\cup\{(0,0,0)\})$.  The flow of 
$X$ defines a smooth embedding $\alpha\colon S^2\times(-1,1)\to B$ which satisfies  $\alpha(\sigma,0)=\alpha_0(\sigma)$.  Moreover 
$$k(\alpha(\sigma,0))=0\text{ and }\tfrac{d}{dt}(k(\alpha(\sigma,t)))=\nabla k\cdot \tfrac{d}{dt}\alpha=\tfrac{|\nabla k|^2}{|\nabla k|^2}=1,
$$
and therefore $k(\alpha(\sigma,t))=t$.
 
 Let $\sigma\in S^2$.  If the gradient flow line through $\alpha_0(\sigma)$  does not limit to a corner of $\partial B$, then $\alpha(\sigma,t)$ extends smoothly over a neighborhood of $ \{\sigma\}\times \{1\}$, since
$\nabla k$ is non-zero along the boundary except at the corners.  On the other hand, if the gradient flow line through $\alpha_0(\sigma)$  limits to a corner, then defining $\alpha(\sigma,1)$ to be this limiting point defines a {\em continuous} bijective extension of $\alpha$ to  $S^2\times (-1,1]$. It follows that this extension is a homeomorphism.
\end{proof}

\subsection{Proof of Theorem \ref{thm7} }\label{POT8}
Since $r,s\in SU(2)$ satisfy $[r,s]=-1$ if and only if $0=\Real(r)=\Real(s)=\Real(r\bar s)$, it follows that the function
 $$\RR P^3\cong SU(2)/\{\pm 1\}\to \chi^*(F_2),~ A\mapsto\left(r_-\mapsto\bbi,~s_-\mapsto\bbj,~ r_+\mapsto A\bbj\bar A, ~s_+\mapsto A\bbi\bar A \right)
 $$
 is a smooth embedding with image $\chi(F_2)_{-1}$.   
  Notice that  $\chi(F_2)_{-1}$ is also equal to the fiber  of $W\circ \phi_-\colon \chi(F^2)\to B$ over $(0,0,0)$.
The map $W\colon \chi(F_-)_{[-1,0]}\to B$ is a diffeomorphism onto its image, the smooth closed 3-ball $B_{[-1,0]}:=\{k\leq 0\}$, by Propositions \ref{prop1} and \ref{GFLonB}. 
 
Part 2 of Lemma \ref{WeilorWeyl} implies that $W\circ \phi_-\colon \chi(F_2)_{[-1,0]}\to B_{[-1,0]}$ is a smooth fiber bundle  with with fiber $\RR P^3$.    Hence $\chi(F_2)_{[-1,0]}$ is diffeomorphic to $\RR P^3\times D^3$, completing the proof of Theorem \ref{thm7}. \qed

\subsection{Further properties of the maps in Diagram (\ref{SVK3})} \label{sectiondiagram}

We next establish several additional properties of the maps in   (\ref{SVK3}) which are used to complete  the proof of Theorem \ref{thm8} in the next subsection.

Consider the $U(1)$ action on $\Hom(\pi_1(F_2),SU(2))_{(-1,1)}$,  known as the {\em modified Goldman twist flow associated to the separating curve $C$} \cite{goldman-invariant,jeffrey-weitsman},
 given  by the formula   
$$e^{\theta\bbi}\cdot \rho\colon  r_-\mapsto \rho(r_-), ~s_-\mapsto \rho(s_-),
 r_+\mapsto e^{\theta I(\rho)} \rho(r_+)e^{-\theta I(\rho)}, ~s_+\mapsto e^{\theta I(\rho)}\rho(s_+)e^{-\theta I(\rho)},
$$
where $$I(\rho)=\tfrac{\Ima(\rho([r_-,s_-])}{\sqrt{1-\kappa(\rho)^2} }\in S^2.$$
It is easy to verify  \cite[Theorem 4.5]{goldman-invariant} that the modified twist flow descends to a free  $U(1)/\{\pm 1\}$ action on $\chi(F_2)_{(-1,1)}$ whose orbits are the fibers of $\phi$.   
\begin{proposition} \label{deadly}
The restriction $\phi\colon \chi(F_2)_{(-1,1)}\to \chi(F_-)\times_{\kappa} \chi(F_+)_{(-1,1)}$ 
 is a smooth principal $S^1$ bundle with respect to the modified Goldman twist flow
 associated to the separating curve $C$.\end{proposition}
\begin{proof} Since $\chi(F_2)_{(-1,1)}$ is a smooth 6-manifold, $\chi(F_-)\times_\kappa \chi(F_+)_{(-1,1)}$ is a smooth 5-manifold,  the $U(1)/\{ \pm 1\}$ action is free, and the fibers of $$\phi=\phi_-\times \phi_+\colon \chi(F_2)_{(-1,1)}\to \chi(F_-)\times_\kappa \chi(F_+)_{(-1,1)}$$
coincide with orbits, the proposition follows from the second part of  
Lemma \ref{WeilorWeyl}.      
\end{proof}
 
 The fibers of $\phi\colon\chi(F_2)\to \chi(F_-)\times_{\kappa}\chi(F_+)$ 
are identified in the following proposition.

\begin{proposition} \label{fibers2}
The fiber of   $\phi\colon \chi(F_2)\to \chi(F_-)\times_{\kappa} \chi(F_+)$
\begin{enumerate}
\item is diffeomorphic to   $S^1$ over any point $(\rho_-,\rho_+)$ with $-1<\kappa(\rho_-)<1$, 
\item is diffeomorphic to  $\RR P^3$ over the the unique  point  $(\rho_0,\rho_0)$ satisfying $\kappa=-1$, where $\rho_0$ is the representation defined in Proposition \ref{prop1}, and
\item  is homeomorphic to  a point or an interval over any point $(\rho_-,\rho_+)$  satisfying $\kappa(\rho_-)=1$, according to whether or not one of $\rho_-,\rho_+$ is central.

\end{enumerate}
\end{proposition}

\begin{proof} 
We apply \cite[Theorem 3.4]{JW2} or \cite[Lemma 4.2]{HHK} which identifies gluing parameters in terms of double cosets of centralizers in $SU(2)$.

The first statement follows from Proposition \ref{deadly}, but we illustrate the double coset approach as a second proof.  If $-1<\kappa<1$, then $\Stab(\rho(C))  \cong U(1)$ and $\Stab(\rho_{\pm} )= \{\pm 1\}$, so the fiber is a circle $\{\pm 1\} \backslash U(1) / \{\pm 1\} \cong U(1)/\{\pm 1\}\cong S^1$.   

If $\kappa =-1$, then both $\rho_+$ and $\rho_-$ equal $\rho_0$, which is non-abelian, so $\Stab(\rho_{\pm})=\{\pm 1\}$, but $\Stab(\rho(C))=SU(2)$ since $\rho(C)$ is central.
Hence, the fiber is $\{\pm 1\}\backslash SU(2)/\{\pm 1\}=SO(3)\cong \RR P^3.$ 

If $\kappa=1$, then $\Stab(\rho(C))=SU(2)$ and both of $\rho_\pm$ are abelian.  If at least one of $\rho_\pm$ is central, then the fiber is a point, and if both are non-central, then the fiber is $U(1)\backslash SU(2) / U(1)\cong S^2/U(1) \cong [-1,1].$  
 \end{proof}

Finally, we prove that $\chi(F_-)\times_{\kappa}\chi(F_+)$ is homeomorphic to a cone on $S^2\times S^2$.  Recall the following facts.  Proposition \ref{prop1} defines   homeomorphisms $W_\pm\colon \chi(F_\pm)\to B$ which are
diffeomorphisms away from the  four corner points. Proposition \ref{GFLonB} shows that $k\circ W_\pm=\kappa$, so that $W\times W$ induces a homeomorphism $\chi(F_-)\times_{\kappa}\chi(F_+)\to B\times _k B$.   A homeomorphism
 $\alpha\colon  S^2\times (-1,1]\to B_{(-1,1]}$, which satisfies $k\circ \alpha (\sigma,t)=t$ and  is a diffeomorphism except at the corners, is   constructed in  Corollary \ref{GF2}.

 Combine these to define 
\begin{equation*}\label{WHA}
\Gamma\colon \chi(F_-)\times_\kappa\chi(F_+)_{(-1,1]}\to S^2\times S^2\times(-1,1]
\end{equation*}
by $$\Gamma(\rho_-,\rho_+)= \big( (\alpha^{-1}(W(\rho_-)),  \alpha^{-1}(W (\rho_+)),\kappa(\rho_-)\big).$$
Since $\chi(F_-)\times_\kappa\chi(F_+)_{-1}$ consists of the single point $(\rho_0,\rho_0)$, the proof of following lemma is clear.

 \begin{lemma}\label{JAB} $\Gamma$ is a homeomorphism, and its restriction  $$\Gamma\colon \chi(F_-)\times_\kappa\chi(F_+)_{(-1,1)}\to S^2\times S^2\times(-1,1),$$ is a diffeomorphism.  
The compact 
space
 $\chi(F_-)\times_{\kappa}\chi(F_+)$ is homeomorphic to  {\rm Cone}$(S^2\times S^2)$.
\end{lemma}

 Set
\begin{equation}\label{WHA1}
\Lambda=  \Gamma \circ \phi\colon  \chi(F_2)_{(-1,1]}\to S^2\times S^2\times(-1,1]
\end{equation}
and, for fixed $t\in(-1,1]$,
\begin{equation}\label{WHA2}
\Lambda_t=  \Gamma \circ \phi\colon  \chi(F_2)_{t}\to S^2\times S^2\times\{t\}\overset{\rm Proj}{\underset{\cong}\longrightarrow} S^2\times S^2
\end{equation}

Proposition \ref{deadly} implies the following: 
\begin{lemma}\label{principal} 
 The restriction
\begin{equation*}
\Lambda\colon \chi(F_2)_{(-1,1)}\to S^2\times S^2\times (-1,1) 
\end{equation*}
is a  smooth principal $S^1$ bundle.  
The fibers of $\Lambda_1\colon \chi(F_2)_1\to S^2\times S^2$ coincide with the 
 fibers of $\phi\colon \chi(F_2)_1\to \chi(F_-)\times_\kappa\chi(F_+)_1$, identified as single points or closed intervals in Proposition \ref{fibers2}.  
 \end{lemma}

\subsection{The proof of Theorem \ref{thm8}} \label{7.6}

 Fix a connection on the smooth $S^1$ bundle $\Lambda$ in Lemma \ref{principal}.  Define a smooth retraction
  $R\colon \chi(F_2)_{(-1,1)}\to \chi(F_2)_{0}$, by taking $R(\rho)$ to be the endpoint of the horizontal lift, starting at $\rho$,  of the smooth segment $(\sigma _-,\sigma _+, (1-u)t), u\in [0,1]$ where $\Lambda(\rho)=(\sigma _-,\sigma _+,t)$. The following diagram commutes. 
  \[
  \begin{tikzcd}
  \chi(F_2)_{0}\arrow[d,swap,"\Lambda_0"] & \chi(F_2)_{(-1,1)}\arrow[d," \Lambda"]\arrow[l,swap,"R"]
  \arrow[r,hook]  &    \chi(F_2)_{(-1,1]}\arrow[d,"\Lambda"] \arrow[dr,"\kappa"]&\\
 S^2\times S^2&S^2\times S^2\times (-1,1)\arrow[r,hook]\arrow[l,"{\rm Proj}"]&S^2\times S^2\times (-1,1]  
 \arrow[r, swap,"{\rm Proj}"] &(-1,1] \end{tikzcd}
  \]

  The (open) mapping cylinder of $\Lambda_0$, $M(\Lambda_0)$ is, by definition,  the identification space obtained
  from $\chi(F_2)_0\times (-1,1]$ by identifying points $( \rho ,1)$ with $( \rho' ,1)$ if and only if 
  $\Lambda_0( \rho )=\Lambda_0( \rho' )$.  In particular there is a canonical inclusion 
  $z\colon S^2\times S^2\hookrightarrow M(\Lambda_0)$ which takes $(\sigma _-, \sigma _+)$ to the equivalence class of $( \rho , 1)$ in the mapping cylinder, where
  $ \rho \in \Lambda_0^{-1}(\sigma _-, \sigma _+)$.

 \medskip 
     Define $G\colon \chi(F_2)_{(-1,1]}\to M(\Lambda_0)$ by
\begin{equation}\label{ache}
  G([\rho])=\begin{cases}
  (R(\rho ,\kappa(\rho) )&  \kappa(\rho)<1,\\
  z\big(\Lambda_1(\rho)\big)& \kappa(\rho)=1.
  \end{cases}
\end{equation}
 \begin{proposition} The map $G$ is  continuous and its restriction $G|_{\chi(F_2)_{(-1,1)}}$ is a diffeomorphism onto $\chi(F_2)_0  \times (-1,1)=  M(\Lambda_0)\setminus z(S^2\times S^2)$.   The fibers of $G$  over points in $z(S^2\times S^2)$ coincide with the fibers of $\Lambda_1\colon \chi(F_2)_1\to S^2\times S^2$ and are therefore points or closed intervals.
 
  \end{proposition}
 \begin{proof}  Only the continuity claim needs justification. Any compact real semi-algebraic set is homeomorphic to a finite simplicial complex (\cite[Theorem 9.2.1]{BCR}), and hence $\chi(F_2)$ is metrizable.  
If $ \rho_i \in \chi(F_2)_{(-1,1)}$ is any sequence which converges in $\chi(F_2)$ to $ \rho_\infty \in \chi(F_2)_1$, then 
 $\kappa(\rho_i)$ converges to $1$ and 
 ${\rm Proj}\circ \Lambda(\rho_i)$ converges to $\Lambda_1( \rho_\infty )$ in $S^2\times S^2$.
 
 The sequence $R( \rho_i )$ lies in the compact manifold $\chi(F_2)_0$, and hence by passing to a subsequence we may assume that $R(\rho_i)$ converges. Commutativity of the diagram shows that
 $$ \Lambda_0(\lim R(\rho_i))=\lim \Lambda_0(R(\rho_i))=\lim {\rm Proj}\circ \Lambda (\rho_i) = {\rm Proj}\circ \Lambda (\rho_\infty)=\Lambda_1(\rho_\infty),$$
 so that $\lim R(\rho_i)\in \chi(F_2)_0$ lies in the fiber of $\Lambda_0$ over $\Lambda_1(\rho_\infty)$.
Therefore,  $$\lim G(\rho_i)=z(\Lambda_1(\rho_\infty))=G(\rho_\infty).
 $$
 \end{proof}

\begin{corollary} The compact manifold 
 $\chi(F_2)_{[0,1]}$ is homeomorphic to the total space of the closed disk bundle over $S^2\times S^2$  associated to the smooth principal $S^1$ bundle $\Lambda_0\colon \chi(F_2)_0\to S^2\times S^2$,   by a homeomorphism which is a diffeomorphism  near the boundary  $\chi(F_2)_{0}$.
\end{corollary}
\begin{proof}

 The mapping cylinder of a principal $S^1$ bundle is easily seen, using polar coordinates, to be homeomorphic to  the associated $D^2$ bundle; in particular, $M(\Lambda_0)$ is a manifold.  Hence $G$ is a continuous, proper map from the manifold $\chi(F_2)_{(-1,1]}$ to the manifold  $M(\Lambda_0)$
 whose fibers are either singletons or closed intervals.  Any pair of compact semi-algebraic sets $(X,Y)$ 
 can be triangulated as a finite simplicial complex and subcomplex (loc. ~cit.), and hence each interval fiber of $G$ admits a contractible regular neighborhood.  In other words,  $G$ is a {\em CE-map}.    Siebenmann's theorem \cite[Approximation Theorem A]{siebenmann}, taking $\ep(x)=\kappa(x)-0.5$ on $\chi(F_2)_{(0.5,1]}$, implies that  there exists a homeomorphism $G'\colon \chi(F_2)_{[0,1]}\to M(\Lambda_0)$ which agrees with $G$ on $\chi(F_2)_{[0, 0.5]}$.
\end{proof}

   To complete the proof of  Theorem \ref{thm8}, it remains only to show that 
the principal $S^1$ bundle $\Lambda_0\colon \chi(F_2)_0\to S^2\times S^2$ has the same Euler class as the unit circle bundle  $S(E)=\partial E\to S^2\times S^2$. 

\medskip

There is an orientation preserving  involution of $F_2$ whose action on $\pi_1(F_2)$ is given by $r_\pm, s_\pm \mapsto r_\mp, s_\mp$. This involution  induces a fiber preserving involution of the principal $S^1$ bundle $\Lambda\colon \chi(F_2)_{(-1,1)}\to S^2\times S^2$
covering the involution $(\sigma_-,\sigma_+)\mapsto(\sigma_+,\sigma_-)$ of $S^2\times S^2$.   This shows that that the Euler class
$e(\Lambda_0)=(n,n)$ for some integer $n$. If needed, reverse the orientation of both  $S^2$ factors so that $n\ge 0$.

The classifying map $c\colon S^2\times S^2\to BS^1$ for $\Lambda_0$ induces  a long exact sequence of homotopy groups \cite{DaK}
$$ \dots \to \pi_2(\chi(F_2)_0)\to \pi_2(S^2\times S^2)\xrightarrow{c_*} \pi_2(BS^1)\to \pi_1(\chi(F_2)_0)\to\pi_1(S^2\times S^2)\to \dots$$
Then $\pi_1(S^2\times S^2)=0$, and $e(\Lambda_0)$ corresponds to $c_*$ under the natural isomorphism $H^2(S^2\times S^2)\cong\Hom(\pi_2(S^2\times S^2),\pi_2(BS^1))$. Therefore 
$\pi_1(\chi(F_2)_0)=\ZZ/g$, where $$g={\rm gcd}\left(e(\Lambda_0)\left[S^2\times\{\sigma_+\}\right], e(\Lambda_0)\left[\{\sigma_+\} \times S^2\right]\right)={\rm gcd}(n,n)=n.
$$ 
On the other hand, $\chi(F_2)_0$ is the boundary of the smooth compact manifold $\chi(F_2)_{[-1,0]}$, which is diffeomorphic to 
$ S^2\times \RR P^3$ by Theorem \ref{thm7}.   Therefore, $n= 2$ and so $S(E)\cong \Lambda_0$, as principal $S^1$ bundles.    This completes the proof of Theorem \ref{thm8}.\qed

 \section{Lagrangians from 3-manifolds with  genus two boundary}\label{lagrange}

Applications of Lagrangian-Floer theory in low dimensional topology often start by
 assigning a symplectic manifold $X(F)$ to a closed 2-manifold $F$ and a Lagrangian immersion $X(Y)\to X(F)$ 
to a 3-manifold $Y$ with boundary $F$. The goal is to construct a  functor $X\colon {\rm Bord}_{2+1}\to {\rm Wein}$ from the bordism category of 2- and 3-manifolds to  the Weinstein  \cite{wein2} $A_\infty$ pre-category (composition  and higher operations are only defined for transverse morphisms, see e.g. \cite[Section 3.4.4]{boklandt})  whose objects are symplectic manifolds and the set of morphisms ${\rm Mor}(M_1,M_2)$ is the set of Lagrangian immersions $L\to M_1^-\times M_2$.   A good overview of this perspective
can be found in Wehrheim's article \cite{W-FFT}.  An explanation of our approach is given in \cite[Sections 10 and 11]{CHKK}.

  Important work  from which this strategy emerged   include Casson's construction of an invariant for homology 3-spheres \cite{Akbulut-McCarthy}, Taubes's gauge-theoretic reformulation of  Casson's invariant \cite{Taubes},   Floer's introduction of Lagrangian Floer homology and instanton homology \cite{floer, Floer}, the Atiyah-Floer conjecture \cite{Atiyah1}, Kronheimer-Mrowka's singular instanton homology \cite{KM1,KM10}, and Oszv\'ath-Szab\'o's Heegard-Floer theory \cite{OS}.

\medskip

In this section we provide some examples of Lagrangian immersions in  $\chi(F_2)$ given by the $SU(2)$ character varieties of 3-manifolds with 
genus two boundary. We work in a general context of {\em webs}, as developed by Kronheimer-Mrowka in \cite{KM10}, which allows our 3-manifolds to contain a link or trivalent graph, and which also allows for non-trivial $SO(3)$ bundles rather than simply the trivial $SU(2)$ bundle. See \cite{CHK}   for a proof (and careful statement) of the folklore assertion that the restriction 
$\chi(Y,T)\to \chi(\partial Y, \partial T)$ is a Lagrangian immersion with respect to the Atiyah-Bott-Goldman form, for $(Y,T)$ a tangle in a compact 3-manifold.

\subsubsection{Notation} To streamline notation, the conjugacy class of $\rho\in\Hom(\pi_1(F_2),SU(2))$ given by
$$
\rho(r_-)=R_-, ~\rho(s_-)=S_-, ~\rho(r_+)=R_+, ~\rho(s_+)=S_+
$$
 (where $R_-,S_-,R_+,S_+\in SU(2)$ satisfy $[R_-,S_-][R_+,S_+]=1$) will be simply denoted by 
 $$[\rho]=\left[\left(R_-,S_-,R_+,S_+\right)\right].$$  Also, given any 3-manifold $Y$ whose boundary contains $F_2$, the inclusion $F_2\subset Y$ induces a restriction map which we denote  $i\colon \chi(Y) \to \chi (F_2)$.

\subsection{The genus two handlebody}
The genus two handlebody $H$ has fundamental group free of rank 2. Its character variety
$\chi(H)$ is therefore a copy of the Goldman pillow $B$, and the abelian locus $\chi(H)^\ab $ is its boundary pillowcase.  
With its natural boundary marking, the homomorphism $\pi_1(F_2)\to \pi_1(H)$ is
given by the surjection $$\langle r_-, s_-,r_+,s_+\mid   [r_-,s_-][r_+,s_+]=1\rangle\to   \langle  s_-  , r_+\rangle  .$$
This implies that the image of the embedding 
$i\colon  \chi(H)\hookrightarrow \chi(F_2)$ lies in $\chi(F_2)_1$.

Define $$ \beta_1,\beta_2\colon [-1,1] \to B, ~ \beta_1(x)=(1,x,-x), ~\beta_2(y)=(y,1,-y)  $$
and $$  p\colon B\to [-1,1]^2,~ p(x,y,z)=(x,y).$$
\begin{proposition}\label{HB}
The diagram
\[
\begin{tikzcd}
B \arrow[d,"p"]&\arrow[l,"W"',"\cong"]\chi(H) \arrow[r, hook, "i"]& \chi(F_2)_1\arrow[d,"\phi"]\\
\text{$[-1,1]^2$} \arrow[r,hook," \beta_1 \times\beta_2  "] & \partial B\times \partial B&\chi(F_-)_1\times_\kappa  \chi(F_+)_1\arrow[l,"W \times W"',"\cong"] 
\end{tikzcd}
\]
commutes. The interval fibers of $p$ are  mapped by $i\circ W^{-1}$ onto  (interval and point) fibers of $\phi$,  homeomorphically in the case of interval fibers of $\phi$.  The image of $i$ equals
$$
\{  \left[\left(1,S_-, R_+,1\right)\right] \mid S_-,R_+\in SU(2)\}.  
$$

\end{proposition} 
\begin{proof}  
  In the diagram above, since $\pi_1(H)$ is generated by $s_-, r_+$, the map $W\colon \chi(H) \to B$ sends the conjugacy class of the  representation
$$s_-\mapsto S_-, 
 r_+\mapsto R_+$$
to $(x,y,z)=\left(\Real(S_-), \Real(R_+), \Real\left(S_- \overline{R_+}\right)\right) \in B$.  The map $\phi\circ i$ takes this conjugacy class to 
 $$\left([r_-\mapsto 1, s_- \mapsto S_-],[r_+\mapsto R_+, s_+\mapsto 1]\right)$$ 
 in  $\chi(F_-)\times_\kappa\chi(F_+)_1$, which is then sent by $W\times W$ to $((1,x,-x), (y,1,-y))$.

Since $i$ and $\beta_1\times \beta_2$ are embeddings and $W$ and  $W\times W$ are homeomorphisms, the fibers of $p$ are embedded by $i\circ W^{-1}$ into the fibers of $\phi$.  The endpoints of the interval fibers of $p$ lie in $\partial B$ and hence are sent by $W^{-1}$ to abelian representations;  the interior points are sent to non-abelian representations.  The interval fibers of $\phi$ have abelian endpoints and non-abelian interiors, and therefore $i\circ W^{-1}$ maps the  fibers of $p$ homeomorphically to the fibers of $\phi$. 
 \end{proof}

 \subsection{Holonomy perturbations and Hamiltonian isotopies of $\chi(F_2)$}\label{holp}
 Given a framed embedded circle $P$ in a 3-manifold $Y$ and a parameter $\ep>0$, a variant of $\chi(Y)$, called the {\em $(P,\ep)$-holonomy perturbed character variety} and denoted  $\chi(Y)_{(P,\ep)}$
 can be defined. This is  the character variety description of the perturbed-flat moduli space resulting from the holonomy perturbations used in instanton gauge theory to  perturb the Chern-Simons function,   introduced by Taubes \cite{Taubes}.

For any oriented surface $F$, it is proven in \cite[Theorem 6.3]{HK2} that  the Lagrangian correspondence induced
by   a holonomy perturbation in a cylinder $F\times I$ along an embedded curve $P\times \left\{\tfrac12 \right\}$ is  equal to graph of   the time-$\ep$ {\em Hamiltonian} flow $\mathcal{F}_{\ep,P}$ on $\chi^*(F)$ given by the (unmodified) Goldman flow associated to the character of $P$. 
An explicit formula for calculating $\mathcal{F}_{(P,\ep)}$ is given in  Proposition 6.2 of \cite{HK2}.  
 
\subsection{A Lagrangian correspondence on multicurves in $(S^2,4)$}
In order to produce   Lagrangian immersions into $\chi(F_2)$, we start by computing a Lagrangian correspondence induced by the pair $(H_0, T)$, where $H_0$ is the 3-manifold obtained by removing an interior open 3-ball from the 2-handlebody $H$, and $T$ is a 2-stranded tangle with its four endpoints on the 2-sphere component of the boundary, illustrated in Figure \ref{fig00fig}. Four meridian loops $a,b,c,d$ are indicated.
\begin{figure}[ht!]
\labellist
\pinlabel $b$ at 225 475
\pinlabel $a$ at 345 475
\pinlabel $c$  at 222 345
\pinlabel $d$  at 330 348
\pinlabel $T$  at 410 500
\endlabellist
\centering
\includegraphics[width=3.7in, height=2.3in]{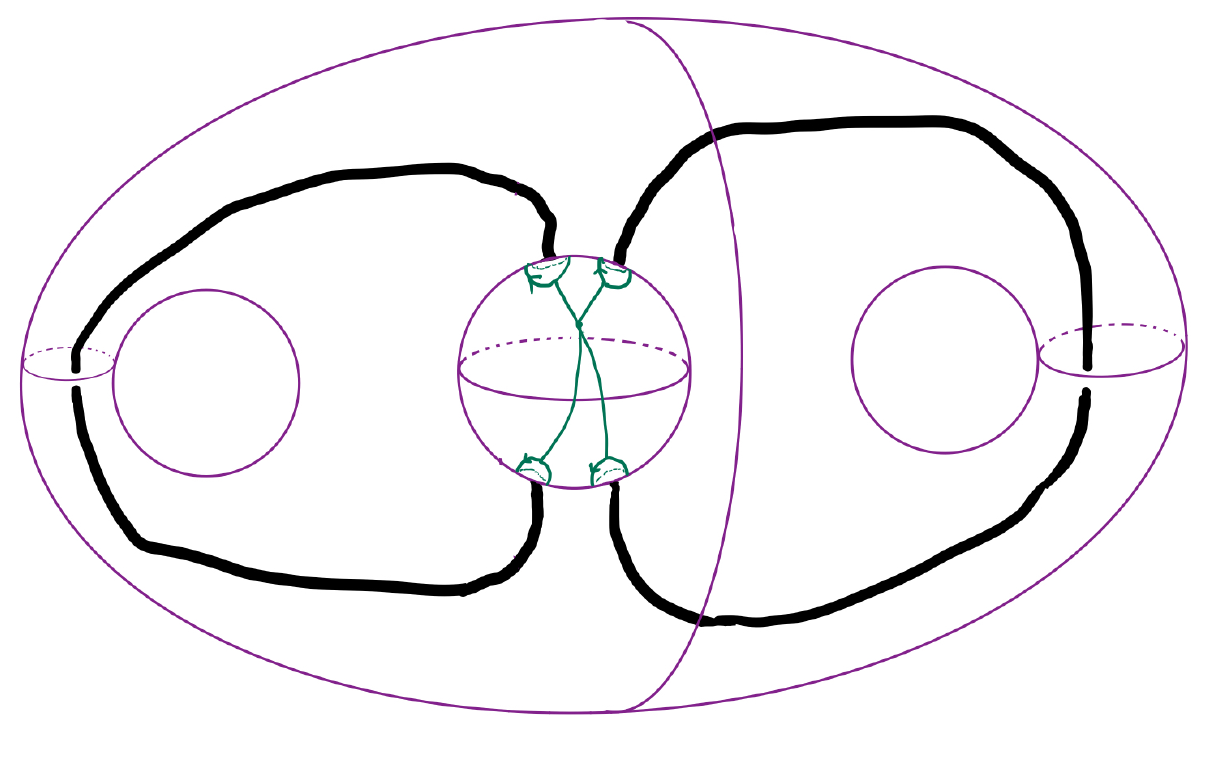}
 \caption{\label{fig00fig} A tangle cobordism from $(S^2,4)$ to $F_2$}
\end{figure}
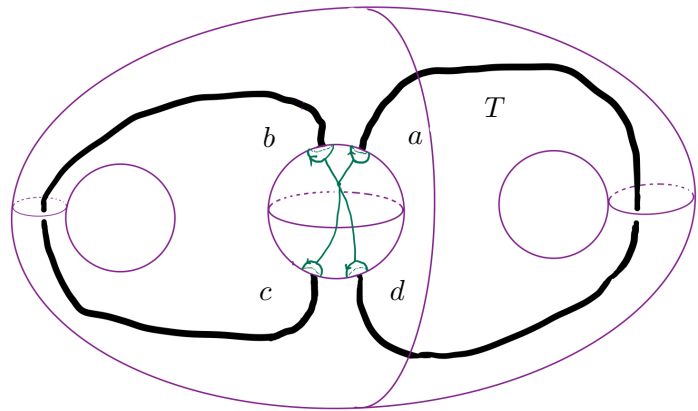
The character variety of the boundary is $\chi(F_2)\times \chi(S^2,4)$, and $(H_0, T)$
determines a restriction map 
\begin{equation}\label{univ}
 i=(i_1,i_2)  \colon \chi(H_0,T)\to \chi(F_2)\times \chi(S^2,4).   
\end{equation}

To view this map as a Lagrangian correspondence, we begin by identifying
$\chi(H_0,T)$.   Note that $H_0\setminus T$ deformation retracts to $F_2$. 
It follows that the restriction $i_1\colon \chi(H_0,T)\to \chi(F_2)$ 
is an embedding with image  
$$i_1\colon \chi(H_0,T) \cong \{ \Real(\rho(r_-))=0,  \Real(\rho(s_+))=0\}\subset \chi(F_2).$$

From Figures \ref{fig0fig} and \ref{fig00fig}, one easily calculates that
\begin{equation}\label{rels0}  r_-=\bar b, ~ s_- b \bar s_-= c, ~s_+=a,~ \text{ and } r_+ a \bar r_+=d.\end{equation}

 Let $G=\langle \sigma_0, \sigma_1, \sigma_2 \mid \sigma_0 ^2, [\sigma_1,\sigma_2], \sigma_1 \sigma _0 \sigma_1 \sigma_0, \sigma_2 \sigma_0\sigma _2 \sigma_0 \rangle$ (the semi-direct product of $\ZZ^2$ and $\ZZ/2$) act  on $\RR^2$ by $$\sigma_0(\gamma, \theta) = (-\gamma, -\theta)
,~ \sigma_1(\gamma, \theta)=(\gamma+2\pi, \theta),~\sigma_2(\gamma,\theta)=(\gamma, \theta+2\pi) .$$   
It is explained in   \cite{HHK}  that $\chi(S^2,4)$
is a pillowcase, homeomorpic to the image of the quotient map $\xi\colon\RR^2 \to \RR^2/ G \cong \chi(S^2,4)$,    
$$\xi\colon\TT^2\to \chi(S^2,4), ~ \xi(\gamma,\theta)=  [a\mapsto \bbi, ~b\mapsto e^{\gamma\bbk}\bbi
, ~c\mapsto e^{\theta\bbk}\bbi, ~d\mapsto e^{(\theta-\gamma)\bbk}\bbi].
$$
Extend this $G$ action on $\RR^2$ to an action on $\RR^2 \times \TT^2$ by 
  $$\sigma_0 (\gamma, \theta, \alpha, \beta) = (-\gamma, -\theta, \alpha+\theta, \beta+\theta),~ \sigma_1(\gamma,\theta,\alpha,\beta)=(\gamma+2\pi, \theta, \alpha+\pi, \beta+\pi),~ \sigma_2(\gamma, \theta, \alpha, \beta)=(\gamma, \theta+2\pi, \alpha, \beta).$$

Define 
\begin{equation}\label{ell1} L\colon \RR^2\times\TT^2\to \chi(F_2),~ L(\gamma, \theta ,  \alpha , \beta )=
\left[\left(-e^{\gamma\bbk}\bbi, 
e^{ \tfrac{\theta-\gamma}{ 2} \bbk}e^{\left(\tfrac{  \theta}{ 2}+\alpha\right) e^{\gamma\bbk}\bbi}
,
e^{ \tfrac{\theta-\gamma} {2} \bbk}e^{\left(\tfrac{  \theta}{ 2} +\beta\right)\bbi}
, \bbi\right)\right].
\end{equation}
We leave to the reader the verifications that  $L$ is $G$ invariant and  $$L(\gamma, \theta ,  \alpha , \beta )([r_-,s_-])=e^{(\gamma-\theta)\bbk}=
L(\gamma, \theta ,  \alpha , \beta )\left(\left[r_+,s_+\right]^{-1}\right).$$

Set 
$$(\RR^2)^*=\RR^2\setminus (\pi\ZZ)^2, ~(\RR^2)^\dag=\RR^2\setminus (\{ (0,\pi), (\pi,0)\}+(2\pi\ZZ)^2),$$
$$ \chi^\dag(S^2,4)= \chi(S^2,4)\setminus\{\xi(0,\pi),\xi(\pi,0)\} \text{ and } \chi^\dag(H_0,T)=i_2^{-1}(\chi^\dag(S^2,4)).$$
Then 
$$
\xi^{-1}(\chi^*(S^2))=(\RR^2)^*
\text{ and } \xi^{-1}(\chi^\dag(S^2))=(\RR^2)^\dag.$$
Recall that $\chi(S^4)$ is homeomorphic to a 2-sphere, so that $\chi^\dag(S^2,4)$ is homeomorphic to an annulus.  The circle $\xi(\tfrac\pi 2,\theta),~\theta\in [0,2\pi]$ is a deformation retract of $\chi^\dag(S^2,4)$.

The map $L$ factors through $(\RR^2\times \TT^2)/G$. In the first part of the following theorem, we prove that the image of $L$ coincides with the  image of the embedding $i_1$.  The rest of the statement refers to the  procrustean commutative diagram into which we insert $\circ= *,~ \dag$, or $ \emptyset$:
\begin{equation}\label{pocus}
\begin{tikzcd}
(\RR^2)^\circ\times \TT^2\arrow[d]\arrow[rd,bend left=25, "L"']&&&& \\
((\RR^2)^\circ \times \TT^2)/G\arrow[r,"L_G"]\arrow[d, "{\rm Proj}"]&{i_1(\chi^\circ(H_0,T))}&&&\chi^\circ(H_0,T)\arrow[d, "i_2"]\arrow[lll, "\cong","i_1"']\\
(\RR^2)^\circ/G\arrow[rrrr, "\xi","\cong" ']&&&&\chi^\circ(S^2,4).
\end{tikzcd}
\end{equation}

\begin{theorem} \label{wish} \hfill
\begin{enumerate}
\item The embedding $i_1\colon\chi(H_0,T)\hookrightarrow \chi(F_2)$ and the map $L$ have the same image.  

\item The restriction 
$ i_2: \chi^\dag(H_0,T) \to \chi^\dag(S^2 , 4)$ is a trivial $\TT^2$ bundle over the annulus.   
\item The restriction 
$ i_2\colon  \chi^*(H_0,T) \to \chi^*(S^2 , 4)$ is smooth.  
\item $L_G$ embeds $(\RR^2)^\dag\times \TT^2/G$ with image $ i_1 \left(  \chi^\dag(H_0,T)\right) $. 
\item The fiber of $i_2$ over each of the points $\xi(\pi,0)$ and $\xi(0,\pi)$ is a circle, and the fiber of $  i_1^{-1}\circ L_G $ over any point in each of these circles is again a circle. Explicitly,
$L(0,\pi,\alpha,\beta)=L(0,\pi,\alpha+\mu,\beta+\mu)$ and $L(\pi,0,\alpha,\beta)=L(\pi,0,\alpha+\mu,\beta-\mu)$ for any $\mu$.  

\item $ \chi^\anc(H_0,T) $ is the union of the two torus fibers of $i_2$ over $\xi(0,0)$ and $\xi(\pi,\pi)$ and $ \chi^\cen(H_0,T) $ is empty.
\end{enumerate}
\end{theorem}
 
\begin{proof}
\noindent{\em (a) The maps $i_1$ and $L$ have the same image.} If $\rho$ lies in the image of  $i_1\colon \chi(H_0,T)\hookrightarrow \chi(F_2)$,
set $\gamma=\arccos(\Real(\rho(r_-s_+)))$. Then $\rho$
can be conjugated so that $\rho=\left[\left(-e^{\gamma\bbk}\bbi, 
S_-,R_+, \bbi\right)\right]$.
This implies $\rho(b)=\rho(\overline{ r_-})=e^{\gamma\bbk}\bbi$ and $\rho(a)=\rho(s_+)=\bbi$.
Equation \ref{rels0} shows that  $S_- e^{\gamma\bbk}\bbi\overline{S_-}=\rho(c)$ and $R_+\bbi\overline{R_+}=\rho(d)$.  The  relation $ba=cd$ implies that, perhaps after further conjugation by $e^{\nu\bbi}$ if needed when $\sin\gamma=0$,   there exists a $ \theta\in   \RR  $ satisfying $\rho(c)=e^{\theta\bbk}\bbi$ and $\rho(d)=e^{(\theta-\gamma)\bbk}\bbi$. 

From (\ref{rels0}) we see that $$S_- e^{\gamma\bbk}\bbi \overline{S_-}=e^{\theta\bbk}\bbi \text{ and }R_+  \bbi \overline{R_+}=e^{(\theta-\gamma)\bbk}\bbi,$$
which can be rewritten   as
 $$  \left( e^{-\tfrac \theta 2\bbk}S_- e^{\tfrac \gamma 2 \bbk}\right) \bbi = \bbi \left( e^{-\tfrac \theta 2\bbk}S_- e^{\tfrac \gamma 2 \bbk}\right) \text{ and }
 \left( e^{-\tfrac{ \theta-\gamma}{2}\bbk}R_+\right )  \bbi =\bbi \left( e^{-\tfrac{ \theta-\gamma}{2}\bbk}R_+\right ) .$$ 
The first equation implies that $e^{-\tfrac{  \theta} {2}\bbk}S_- e^{\tfrac \gamma 2 \bbk}=e^{x\bbi}$ for some $x$; the second that  $e^{-\tfrac{ \theta-\gamma}{2}\bbk}R_+=e^{y\bbi}$ for some $y$.
Therefore
$$S_-=e^{\tfrac{  \theta} {2}\bbk}e^{x\bbi}e^{-\tfrac \gamma 2 \bbk}
=e^{\tfrac{  \theta} {2}\bbk}e^{-\tfrac \gamma 2 \bbk}\left(e^{\tfrac \gamma 2 \bbk}
e^{x\bbi}e^{-\tfrac \gamma 2 \bbk}\right)=e^{-\tfrac \gamma 2 \bbk}e^{\tfrac{  \theta}{2}\bbk}e^{xe^{\gamma\bbk}\bbi},
\text{ and similarly }  R_+=e^{\tfrac{ \theta-\gamma}{2}\bbk}e^{y\bbi}.$$ The substitutions
$\alpha=x-\tfrac{ \theta}{2},~ \beta=y-\tfrac{ \theta}{2}$ then show
$\rho=L(\gamma, \theta,\alpha,\beta)$.  Hence the image of $L$ contains 
the image of $i_1$.  
The reverse inclusion follows immediately from $\Real \left(L(\gamma, \theta,\alpha,\beta)(r_-)\right)=0$ and $\Real \left(L(\gamma, \theta,\alpha,\beta)(s_+)\right)=0$.

\medskip

\noindent{\em (b) The fibers of $i_2$.}    Fix  $(\gamma,\theta)\in \RR^2$.  
Suppose that $\rho\in \chi(H_0,T)$ satisfies $i_2(\rho)=\xi(\gamma,\theta)$.
From the previous paragraph we can write $i_1(\rho)=L(\gamma',\theta', \alpha,\beta)$.
Then $\xi(\gamma,\theta)=\xi(\gamma',\theta')$, which means there is a sign $\ep\in \{\pm 1\}$ and $\ell,n\in \ZZ$ so that $(\gamma',\theta')=\ep(\gamma,\theta) +2\pi(\ell,n)$.
Hence $i_1(\rho)=L(\ep \gamma+2\pi\ell, \ep\theta+2\pi n, \alpha,\beta)$, which equals   $L(\gamma,\theta,\alpha+\pi \ell, \beta+\pi \ell)$ if $ \ep=1$ and  $L(\gamma,\theta,\alpha+\pi \ell - \theta, \beta+\pi \ell - \theta)$ if $\ep=-1,$
  using the $G$ invariance of $L$.  It follows that  the fiber of $i_2$ over $\xi(\gamma,\theta)$ equals  $L(\{(\gamma,\theta)\}\times \TT^2)$.

\medskip

Suppose that $L(\gamma,\theta,\alpha,\beta)=L(\gamma,\theta,\alpha',\beta')$.
Then there exists a $0\leq \nu<\pi$ so that
$$
\left(-e^{\gamma\bbk}\bbi, 
e^{ \tfrac{\theta-\gamma}{ 2} \bbk}e^{\left(\tfrac{  \theta}{ 2}+\alpha \right) e^{\gamma\bbk}\bbi}
,
e^{ \tfrac{\theta-\gamma} {2} \bbk}e^{\left(\tfrac{  \theta}{ 2} +\beta\right)\bbi}
, \bbi\right)
=
e^{\nu\bbi}\left(-e^{\gamma\bbk}\bbi, 
e^{ \tfrac{\theta-\gamma}{ 2} \bbk}e^{\left(\tfrac{  \theta}{ 2}+\alpha'\right) e^{\gamma\bbk}\bbi}
,
e^{ \tfrac{\theta-\gamma} {2} \bbk}e^{\left(\tfrac{  \theta}{ 2} +\beta'\right)\bbi}
, \bbi\right)e^{-\nu\bbi}.
$$
If $\sin\gamma\ne 0$, then $\nu=0$,  from consideration of the first component. 
 Therefore, $(\alpha,\beta)=(\alpha',\beta')$, and so $L$ embeds $\{(\gamma,\theta)\}\times \TT^2$   onto the fiber of $i_2$ over $\xi(\gamma,\theta)$ if $\sin\gamma\ne 0$.

It remains to consider the case when $\gamma=\ell\pi$
with $\ell\in \ZZ$.  
Then 
$L(\ell\pi,\theta, \alpha,\beta)
   =L(\ell\pi,\theta, \alpha',\beta') $ if and only if 
\begin{equation}\label{fpaint}
\left( e^{\tfrac{  \theta-\ell \pi}{ 2} \bbk}
e^{(-1)^\ell\left( \tfrac{  \theta}{ 2}+\alpha\right)  \bbi} ,
 e^{\tfrac{  \theta-\ell \pi}{ 2} \bbk}
e^{\left(\tfrac{  \theta}{ 2} +\beta\right) \bbi} \right)
=e^{\nu\bbi}
\left(   e^{\tfrac{  \theta-\ell \pi}{ 2} \bbk}
e^{(-1)^\ell\left( \tfrac{  \theta}{ 2}+\alpha'\right)  \bbi} ,
 e^{\tfrac{  \theta-\ell \pi}{ 2} \bbk}
e^{\left(\tfrac{  \theta}{ 2} +\beta'\right) \bbi} \right)e^{-\nu\bbi}
 \end{equation}
 for some $\nu\in[0,\pi)$.  After some rearrangement, (\ref{fpaint}) is equivalent to 
\begin{equation}\label{fine artiste}
 e^{-\tfrac{\theta-\ell \pi}{2} \bbk} e^{\nu \bbi} e^{\tfrac{\theta-\ell \pi}{2} \bbk} = e^{((-1)^{\ell}(\alpha - \alpha') +\nu) \bbi} = e^{(\beta-\beta' +\nu) \bbi },\end{equation} so, assuming $(\alpha,\beta)\neq (\alpha',\beta')$, it must be the case that $e^{\tfrac{\theta-\ell \pi}{2} \bbk} =\pm k$.  This means that $\tfrac{\theta - \ell \pi}{2} $ is an odd multiple of $\tfrac{\pi}{2}$, i.e., $\gamma=\ell \pi $ and $\theta=n\pi$ for integers $\ell, n$ with different parity.  In this case, 
 $$\alpha=\alpha'+(-1)^{\ell} 2\nu\text{ and } \beta =\beta' +2\nu.$$
 
In summary,  the fiber of $i_2\colon \chi(H_0,T)\to \chi(S^2,4)$ over $[\theta,\gamma]$ is
 a torus unless $\gamma=\ell \pi$ and $\theta=n\pi$ with exactly one of $\ell,n$ odd,   in which case the fiber is the circle, $\TT^2$ modulo the diagonal or antidiagonal subgroup.

\medskip

\noindent{\em (c) The map $i_2$ is a bundle over $\chi^\dag(S^2,4)$.}  
 Consider the diagram (\ref{pocus}), in the three cases of $\circ=*,\dag,\emptyset$. 
 Part (a) of the argument (above) shows that $L$ is onto in all three cases. Part (b)   shows that $L_G$ is a diffeomorphism for $\circ=*$, a homeomorphism for $\circ=\dag$, and the   fiber of $L$ over each point $i_1(\rho)$ such that $i_2(\rho)\in \{\xi(0,\pi),\xi(\pi,0)\}$ is a circle.

If $\circ=*$, the left vertical map in the diagram is a smooth $\TT^2$ bundle,  since $G$ acts freely on $(\RR^2)^*$. Hence the right vertical map is a smooth $\TT^2$ bundle over $\chi^*(T^2,4)$.

For the  $\circ=\dag$ statement,    
consider the normal subgroup $\ZZ^2 \subset G$ generated by $\sigma_1,\sigma_2$.  
$\ZZ^2$ acts freely on $(\RR^2)^\dagger$, with quotient a twice punctured torus parameterized by $(\gamma, \theta)$.  So $((\RR^2)^\dagger \times \TT^2)/\ZZ^2 $ is a torus bundle over $(\RR^2)^\dagger / \ZZ^2$.  From this torus bundle, we can obtain $((\RR^2)^\dagger \times \TT^2)/G $ by further modding out by the action of $G/\ZZ ^2 \cong \ZZ_2$.  But note that, while $\sigma_0$ acts nontrivially on the torus factor of $\RR^2\times \TT^2$,   $[\sigma_0]\in G/\ZZ^2$ acts trivially on the fibers of $((\RR^2)^\dagger \times \TT^2)/\ZZ^2 $ over the fixed points $[0,0]$ and $[\pi, \pi]$ in the base space, because
$$\sigma_0(0,0,\alpha,\beta)=(0,0,\alpha,\beta), ~
\sigma_0(\pi,\pi,\alpha,\beta)=(-\pi,-\pi,\alpha+\pi,\beta+\pi) = \sigma_2 ^{-1}\sigma_1^{-1}(\pi,\pi,\alpha,\beta).
$$  This   
 implies that $((\RR^2)^\dag \times \TT^2)/G\to (\RR^2)^\dag /G$ is a  $\TT^2$ bundle over the topological annulus $\chi^\dagger (S^2, 4)$.  In fact, it is a trivial $\TT^2$ bundle, because  the annulus deformation retracts to the circle parameterized by $\xi(\tfrac \pi 2, \theta), ~0\leq \theta \leq 2\pi$ and the action of $\sigma_2$ shows that the bundle is trivial over this circle.
  \end{proof}

\subsubsection{Multicurves}
 Let
  ${\rm Curv}(\chi(S^2,4))$ denote the set of immersed {\em multicurves} in $\chi(S^2,4)$.
  By a multicurve we mean a finite union of   immersions $f_i$,  with domains $S^1$ or $[0,1]$, and range $\chi(S^2,4)$, such that
  that circles and interior of arcs are mapped into $\chi^*(S^2,4)$ and endpoints of arcs are mapped to corners, asymptotically linearly (see \cite[Section 6]{HK2}).  See also
  \cite{HHHK, HRW, KWZ} for the use of  multicurves in bordered Heegaard-Floer and Khovanov homology.
Theorem 8.1 of \cite{HK2} 
  shows that,  for any 2-tangle $(B,S)$ in a $\ZZ$-homology ball $B$, one can find a small holonomy perturbation $\pi$  so that  $\chi(B,S)_\pi\in  {\rm Curv}(\chi(S^2,4))$.

\medskip

The Lagrangian correspondence \cite{wein2}
$$\mathcal{L}_{(H_0,T)}\colon {\rm Curv}\left(\chi(S^2,4)\right)\to {\rm Lag}\left(\chi(F_2)\right)$$ induced by
$i\colon \chi(H_0,T)\to \chi(F_2)\times \chi(S^2,4)$ is defined set-theoretically as follows. Given $f\colon J\to \chi(S^2,4)$,    we let $M(f)$ denote the pullback  of $f$ and $i_1$.  
Then   the top horizontal row in the following diagram  defines $\mathcal{L}_{(H_0,T)}(f)$.\begin{equation}\label{diag55}
\begin{tikzcd}
M(f)\arrow[r,"\tilde f"]\arrow[d]&\chi(H_0,K)\arrow[d,"i_1"]\arrow[r,"i_2"]&\chi(F_2)\\
J\arrow[r,"f"]&\chi(S^2,4)&
\end{tikzcd}
\end{equation}
 
  If $J$ is a circle, its image lies in $\chi^*(S^2,4)$, and hence
 the pullback $M(f)$ is diffeomorphic to
 $\TT^3$ and $i_2\circ \tilde f\colon M(f)\to \chi(F_2)$ is a Lagrangian immersion.

 If $J$ is a closed arc, then the part of $M(f)$ which lies over the interior of $J$ is a smooth  open cylinder ${\rm Int}(J)\times \TT^2$.  Theorem \ref{wish} implies  that  $M(f)$ is either $J\times \TT^2$, or obtained from $J\times \TT^2$ by collapsing the torus to a circle over one of both endpoints.

  \begin{corollary}\label{correspond} The Lagrangian correspondence 
$$\mathcal{L}_{(H_0,T)}\colon {\rm Curv}\left(\chi(S^2,4)\right)\to {\rm Lag}\left(\chi(F_2)\right)$$
induced by $
i\colon \chi(H_0,T)\to \chi(F_2)\times \chi(S^2,4)$
acts in the following way. 
\begin{itemize}
\item A   smooth immersion of a circle $f\colon S^1\to \chi^*(S^2,4)$  is sent  by  $\mathcal{L}_{(H_0,T)}$
to the smooth Lagrangian immersion 
$$i_2\circ \tilde f\colon M(f)\cong \TT^3\to \chi^*(F_2).$$

\item An immersion of an arc $f\colon J\to \chi(S^2,4)$ is sent to  
$$i_2\circ \tilde f\colon M(f)\to \chi(F_2),$$ a Lagrangian immersion over $\chi^*(F_2)$,  where $M(f)$ is  one of the following: 
\begin{enumerate}
\item  $M(f)$ is a lens space 
 if both endpoints of  $f$ lie in   $\{\xi(0,\pi),\xi(\pi,0)\}$.
 \item $M(f)$ is 
a cylinder $I\times \TT^2$ if both endpoints of $f$ lie in  $\{\xi(0,0),\xi(\pi,\pi)\}$. In this case the boundary
tori are (affinely) embedded in $\chi^\anc(F_2)\subset \TT^4/\{\pm\}$.
\item  $M(f)$ is a solid torus 
$S^1\times D^2$ if $f$ has one endpoint of each type.  In this case, the boundary
torus is (affinely) embedded in $\chi^\anc(F_2)\subset \TT^4/\{\pm\}$.
\end{enumerate}
\end{itemize}
\end{corollary}

\noindent{\em Remark.}  In the next subsection, we parameterize examples of 
all four possibilities in Corollary \ref{correspond} using maps $[0,2\pi]\times \TT^2\to \chi(F_2)$.

\subsection{Examples}\label{ccm}   
Figure \ref{fig10fig} illustrates several tangles; some with (red) perturbation curves and (blue) $w_2$ arcs. These latter encode gauge-theoretic data required to ensure
smoothness of certain instanton moduli spaces in the construction of instanton homology by Floer and Kronheimer-Mrowka. In particular, Figures \ref{fig10fig} (e) and (f) are instances of {\em adding a strong marking}, a step used by Kronheimer-Mrowka in their construction of the singular instanton homology of links and webs \cite{KM10,KM-khovanov}.

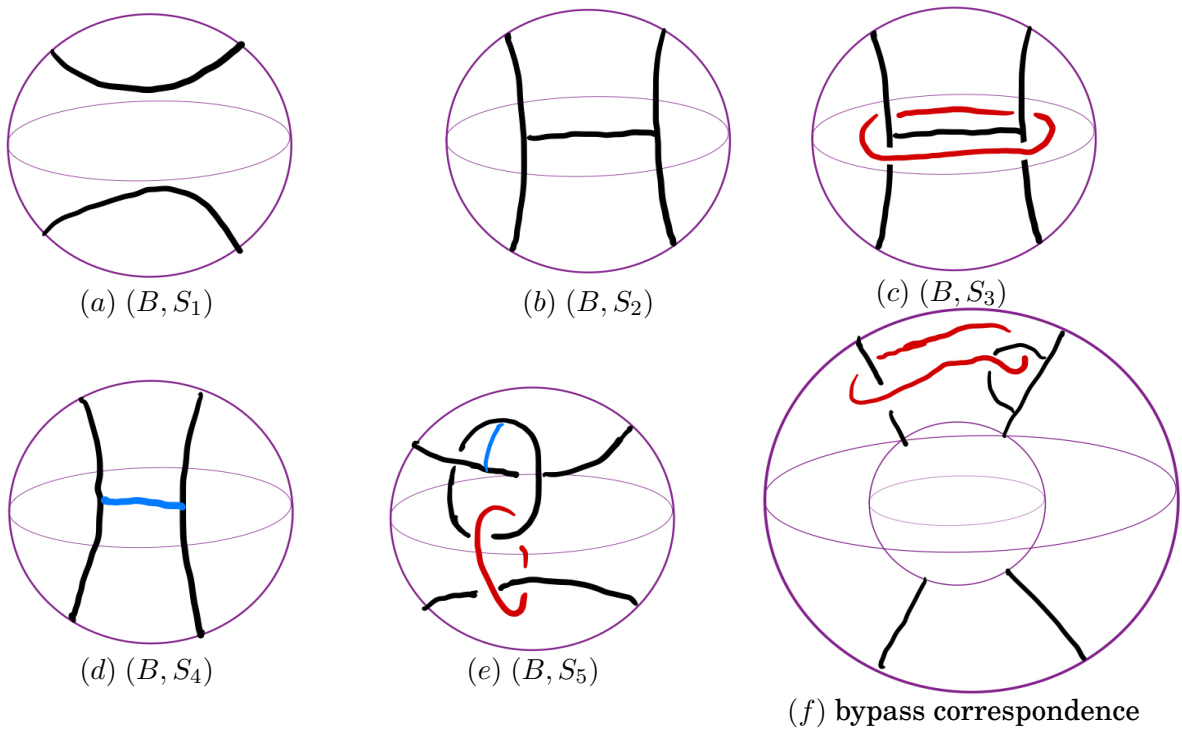
\begin{figure}[ht!]
\labellist
\pinlabel $(a)~(B,S_1)$ at 100 420 
\pinlabel $(b)~(B,S_2)$ at 300 420
\pinlabel $(c)~(B,S_3)$  at 460 425
\pinlabel $(d)~(B,S_4)$  at 100 240
\pinlabel $(e)~(B,S_5)$  at 275 240
\pinlabel $(f)~\text{bypass correspondence}$  at 470 220
\endlabellist
\centering
\includegraphics[width=6.5in, height=3.8in]{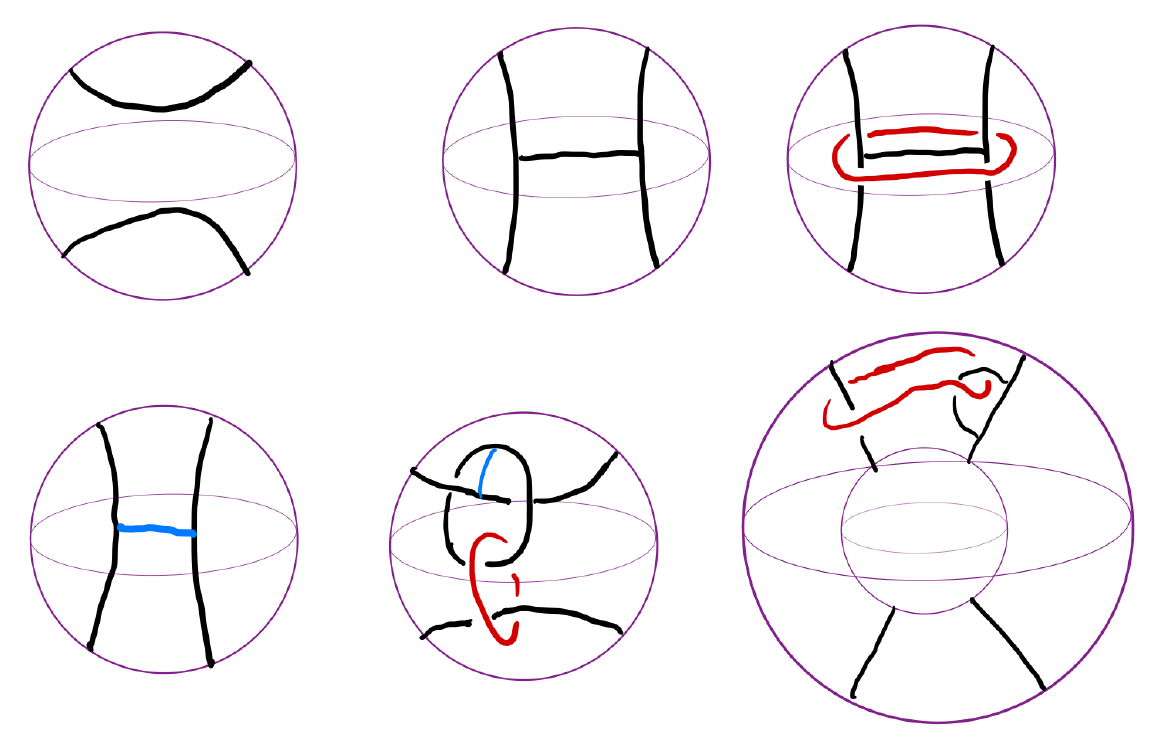}
 \caption{\label{fig10fig} Four-ended tangles in the ball and 
 a modified product tangle in $S^2\times I$, enhanced with perturbation curves and $w_2$ arcs }\end{figure}

\subsubsection{Example (a)}  Consider the knot $ K$ in the handlebody $H$ obtained by
filling in  the 2-sphere boundary component of $H_0$   with the tangle $(B,S_1)$  in Figure \ref{fig10fig} (a).  It is simple to verify that $\chi(B,S_1)$ is an arc, its image in $\chi(S^2,4)$ is parameterized as $\xi(\pi, t),~0\leq t<\pi$.   Corollary \ref{correspond} implies  that $\chi(H,K)\cong D^2\times S^1$. The center circle maps to a circle in $\chi(F_2)_{-1}$ and the boundary torus maps to $\chi(F_2)^\anc$.

\subsubsection{Example (b)}  Consider the {\em eyeglass web} $W$ in the handlebody $H$ obtained by
filling in the boundary four-punctured 2-sphere of $H_0$ with the web $(B,S_2)$ illustrated in Figure \ref{fig10fig} (b) (see also Figure \ref{fig9fig}).  It is simple to verify that $\chi(B,S_2)$ is a circle with image in $\chi^*(S^2,4)$ parameterized as $\xi(t,t+\frac\pi 2),~0\leq t\leq 2\pi$.  Applying Corollary \ref{correspond}, one concludes that $\chi(H,W)$ is an embedded 3-torus parameterized by
$$L(t,t+\tfrac\pi 2,\alpha, \beta),~ (t,\alpha,\beta)\in [0,2\pi]\times \TT^2.$$ 

The image of the Lagrangian embedding  $\chi(H,W)\hookrightarrow\chi^*(F_2)$ is 
 the orbit of the conjugacy class of 
$L(0,\tfrac\pi 2,-\tfrac\pi 4, -\tfrac\pi 4)=\left[\left(-\bbi, e^{\tfrac\pi 4\bbk}, e^{\tfrac\pi 4\bbk},\bbi \right)\right]$
under the  Jeffrey-Weitsman  Hamiltonian $\TT^3$ action \cite{jeffrey-weitsman}.

 \subsubsection{Example (c)}  The tangle $(B,S_3)$ in Figure \ref{fig10fig} (c) again produces the eyeglass web $W$ in $H$ when glued to $H_0$.  The extra (red) circle is a holonomy perturbation curve which we call $P$.  As explained in Proposition 6.2 of \cite{HK2}, since $P$ is isotopic into the 4-punctured $S^2$ boundary component,   for small $\ep$, the holonomy perturbed character variety $\chi(B,S_3)_{P,\ep}$ is diffeomorphic to the unperturbed variety $\chi(B,S_3) =\chi(B,S_2),$ and hence is a circle.  The restriction map $\chi(B,S_3)_{P,\ep}\to \chi(S^2,4)$ has the same image as
the restriction map $\chi(B,S_2)\to \chi(S^2,4)$ composed with the homeomorphism 
$$H_\ep\colon \chi(S^2,4)\to \chi(S^2,4),~H_\ep\left(\xi(\gamma,\theta)\right)= \xi(\gamma,\theta-2\ep\sin\gamma).$$
In particular $\kappa=\cos( \tfrac \pi 2 - 2\epsilon \sin t ),$ which is non-zero if $\ep\sin t\ne 0$. 

 Using  the previous example, and Proposition 6.2 of \cite{HK2}, one  concludes that the embedded Lagrangian torus $i(\chi(H,W)_{P,\ep}))\subset \chi^*(F_2)$ is parameterized by 
 $$(t,\alpha, \beta)\mapsto L(t, t+\tfrac\pi 2 -2\ep \sin t, \alpha, \beta),~ (t,\alpha,\beta)\in [0,2\pi]\times\TT^2.$$  
When $\ep\neq 0$,  this Lagrangian intersects $ \chi(F_2)_0$ in the pair of tori $L(\{(0,\tfrac \pi 2), (\pi, \tfrac{3\pi}{2})\}\times\TT^2)$. Contrast this with the unperturbed case  where $i(\chi(H,W))\subset \chi(F_2)_0$.
 
 \subsubsection{Example (d)}\label{w2}  The tangle $(B,S_4)$  in Figure \ref{fig10fig} (d)
 has two vertical arcs, augmented by a  (blue) {\em $w_2$-arc} $A$, which represents the  Poincar\'e dual of a relative second Stiefel-Whitney class \cite{KM1,HHK}.  The   character variety $\chi(B,S_4, A) $
 can be defined as
  $$\chi(B,S_4,A)=\left. \left\{\rho\in \Hom(\pi_1(B\setminus\{S_4\cup A\}),SU(2))\mid \Real(\rho(\mu_i))=0,~\Real(\rho(\mu_A))=-1\right\}\right/_\conju 
  $$
  where $\mu_i,~i=1,2$ denotes the meridians of the two strands of $S_4$ and $\mu_A$ denotes the
  meridian of the $w_2$-arc $A$.  It is easily seen to be parameterized by  the arc
  $\xi(t, t+\pi), ~t\in[0,\pi].$ Corollary \ref{correspond} implies that 
  $\chi(B,S_4,A) $ is a lens space.  
  
Gluing  $(B,S_4,A)$ to $H_0$  yields a 2-component link $L$ in $H$ augmented by the $w_2$-arc.
  Since $\kappa=-1$,  
  $i\colon  \chi(H,L,A)\to \chi(F_2)$ is a Lagrangian embedding of a lens space into $\chi(F_2)_{-1}\cong \RR P^3$ (Section \ref{POT8}), so $\chi(H,L,A)=\chi(F_2)_{-1}$.   The Darboux-Weinstein  theorem \cite{weinstein} implies that    $\chi(F_2)_{[-1,0]}$  is diffeomorphic to the (trivial) unit cotangent disk bundle of $\RR P^3$, giving a second, symplectic proof of Theorem \ref{thm7}.  
  
  \medskip 
  
   From examples (a) and (d), one can glue $(B,S_1)$ or $(B,S_4)$ by any mapping class element of $(S^2,4)$ to  $H_0$, to produce embedded arcs of different slopes (see \cite{HHK,FKP}), providing ample examples of embedded Lagrangian manifolds in $\chi^*(F_2)$  of types (1), (2), and (3) in the Corollary.
\subsubsection{Example (e)}   The tangle $(B, S_5)$ in Figure \ref{fig10fig} (e) is a pair of horizontal arcs, along with a linking circle on one, augmented by a perturbation curve $P$ and a $w_2$-arc $A$.  
 Gluing this tangle to $H_0$ results in a  knot  $K$ in $H$.  
  
   In \cite[Theorem 7.1]{HHK}, it is shown that $\chi(B,S_5)_{P,\ep}\to \chi(S^2,4)$ is 
 a ``Figure 8'' immersion of a circle, parameterized as $$\xi(\tfrac\pi 2 +t+\ep\sin t  , \tfrac\pi 2 +t-\ep\sin t ), ~ t\in [0,2\pi].$$
By Corollary \ref{correspond}, $\chi(H,K,A)_{P,\ep}\cong  \TT^3$, and  $i(\chi(H,K,A)_{P,\ep})\subset \chi^*(F_2)$ equals the image of
$$(t, \alpha, \beta)\mapsto L(\tfrac\pi 2 +t+\ep\sin t  , \tfrac\pi 2 +t-\ep\sin t ,\alpha,\beta), ~ (t,\alpha, \beta)\in [0,2\pi]\times \TT^2. $$
Since $\kappa=\cos(2\ep \sin t)$,  $\chi(H,K,A)_{P,\ep}$ meets $\chi(F_2)_1$ in the two 2-tori 
 $L(\{\pm (\tfrac\pi 2,\tfrac\pi 2)\}\times\TT^2)$
 for $\ep\neq 0$, unlike the unperturbed variety
 $\chi(H,K,A)$, whose restriction lies in $\chi(F_2)_1.$

 \subsubsection{Example (f)} The  tangled web  $(S^2\times I, E)$ in Figure \ref{fig10fig} (f) is called the {\em bypass correspondence} in \cite{HK3}.  Gluing one $(S^2,4)$ boundary component of $(S^2\times I, E)$ to $(S^2,4)\subset \partial H_0$ results in another tangled web in $H_0$.  Theorem B of \cite{HK3} asserts that  $\mathcal{L}_{(S^2\times I, E)}\colon {\rm Curv}(\chi(S^2,4))\to  {\rm Curv}(\chi(S^2,4))$ 
 takes an immersed circle $S^1\to \chi^*(F_2)$ to its {\em double} (precomposition by $S^1\times\{\pm1\}\to S^1$), and takes an immersed arc $I\to \chi^*(F_2)$ to the {\em Figure 8 curve in $\chi(F_2)$ supported near the immersed arc}  \cite[Definition 6.7]{HK3}.  Thus, the composite Lagrangian correspondence
\begin{equation}\label{cool2}
{\rm Curv}(\chi(S^2,4))\xrightarrow{\mathcal{L}_{(S^2\times I, E)}}  {\rm Curv}(\chi(S^2,4))
 \xrightarrow{\mathcal{L}_{(H_0,T)}} {\rm Lag}(\chi(F_2))\end{equation}
takes any circle  to a pair of immersed 3-tori, and it takes any immersed arc to the 3-torus over the Figure 8 curve   supported near the arc.

\subsubsection{Perturbing a non-transverse  pair of Lagrangians in $\chi(F_2)$.}\label{HOE}
We  make  a pair of Lagrangians  associated to a Heegaard splitting transverse using a holonomy perturbation.
Figure \ref{fig9fig} illustrates a linked pair $W_1,W_2$ of eyeglasses in $S^3$, separated by a genus 2 surface. Also illustrated is a (red) holonomy perturbation  curve $P$ which lies on the Heegaard surface $F_2$.

 We identify the solid handlebody containing $W_1$ with $H$ and call its complement $H'$, so that 
 \begin{equation}\label{LE}
 (S^3,W_1\sqcup W_2)=(H,W_1)\cup_{F_2}(H',W_2).\end{equation}
 The pair $(H,W_1)$ equals the pair $(H,W)$ from Example (b).
 Moreover,   there exists a diffeomorphism $d\colon  (H',W_2)\to (H,W)$ whose
boundary restriction  $d_\partial\colon F_2\to F_2$ satisfies  $d_\partial(r_\pm)=s_\pm$ and $d_\partial(s_\pm) =r_\pm$. \begin{figure}[ht!]
\labellist
\pinlabel ${\color{red} P}$ at 265 360 
\pinlabel $W_1$ at 170 510 
\pinlabel $W_2$ at 140 290 
 \endlabellist
\centering
\includegraphics[width=2.8in] {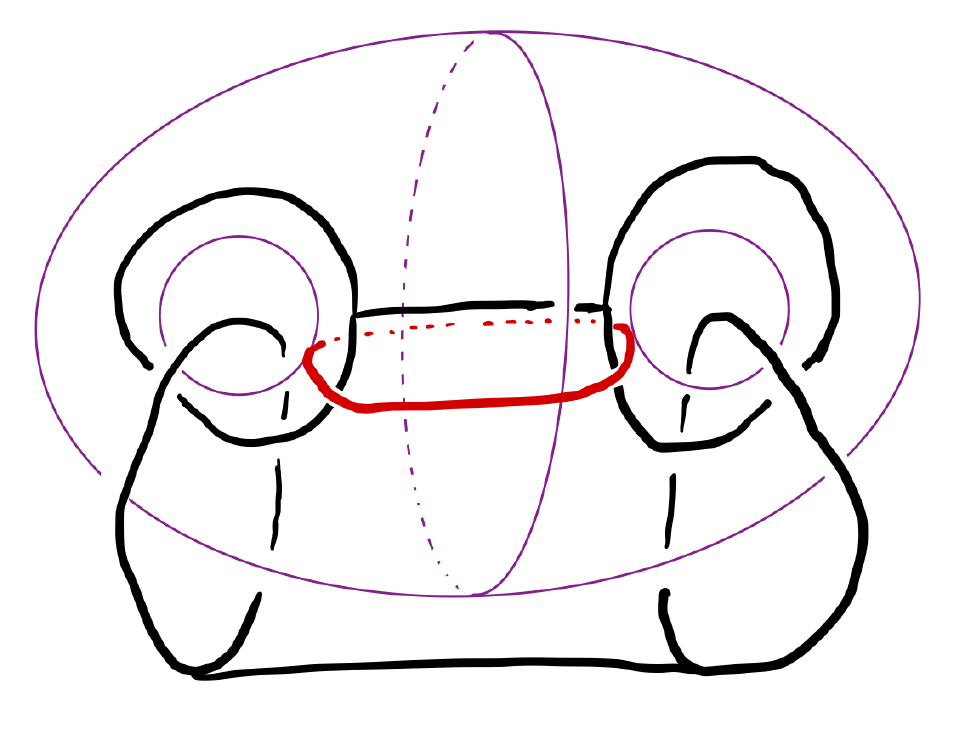}
  \caption{\label{fig9fig} A genus two splitting of a linked pair of eyeglasses}
\end{figure}

Therefore, the decomposition (\ref{LE}) determines the commutative diagram:
\begin{equation*}\label{diag9.1}
\begin{tikzcd}
\chi(S^3,W_1\sqcup W_2)_{P,\ep}\arrow[d]\arrow[rr]&&\chi(H',W_2)\arrow[d,"i_2", hook]&\chi(H,W)\cong \TT^3 \arrow[d,"i_1", hook]\arrow[l," d","\cong"']\\
\chi(H,W)\cong\TT^3\arrow[r, "i_1", hook]&\chi(F_2)\arrow[r,"\mathcal{F}_{P,\ep}", "\cong"']&\chi(F_2)&\chi(F_2)\arrow[l, "d_\partial"', "\cong"]
\end{tikzcd}
\end{equation*}
so the maps $i_2$ and  $d_\partial\circ i_1$ have the same image.
The gluing parameters are trivial, and every map in the diagram is injective. Example (c) shows that the composite $i_{1,\ep}:=\mathcal{F}_{P,\ep}\circ i_1$ has image parameterized
as 
$$i_{1,\ep}(t, \alpha,\beta)=L(t, t+\tfrac \pi 2 -2\ep \sin t,\alpha, \beta),~(t, \alpha,\beta)\in [0,2\pi]\times \TT^2.$$
Using Example (b), the smooth embedding $i_2$ has image parameterized as 
$$d_\partial \circ L(t, t+\tfrac \pi 2,\alpha,\beta),~  (t, \alpha,\beta)\in [0,2\pi]\times \TT^2,$$
where $d_\partial ([R_-,S_-,R_+,S_+])=([S_-,R_-,S_+,R_+]).$

\medskip

We identify the intersection $i_{1,\ep}(\TT^3)\cap i_2(\TT^3)$.
Suppose $i_{1,\ep}(t, \alpha,\beta)\in i_2(\TT^3)$.   
Every $\rho$ in the image of $i_2$ satisfies $\Real(\rho(s_-))=0$ and $\Real(\rho(r_+))=0$.
From the first of these equations, 
$$
0=\Real(i_{1,\ep}(t, \alpha,\beta)(s_-) )= \Real(L(t,  t+\tfrac \pi 2 -2\ep \sin t, \alpha, \beta)(s_-))
=\cos(\tfrac \pi 4 -\ep\sin t)\cos(\tfrac \pi 4 -\ep \sin t +\tfrac t 2 +\alpha).$$
Since $\cos(\tfrac \pi 4 -\ep\sin t)$ is non-zero if $\ep$ is small enough,  $0=\cos(\tfrac \pi 4 -\ep \sin t +\tfrac t 2 +\alpha)$, and therefore 
$$\alpha=\tfrac \pi 4 +\ep \sin t -\tfrac t 2 +\ell_\alpha \pi$$ for $\ell_\alpha=0$ or $1$. Applying the same logic to $r_+$ instead of $s_-$ shows that 
$$\beta=\tfrac \pi 4 +\ep \sin t -\tfrac t 2 +\ell_\beta \pi$$ for $\ell_\beta=0$ or $1$.

If $\ep=0$, then $t,\ell_\alpha, \ell_\beta$ determine $\alpha$ and $\beta$.   The intersection of the {\em unperturbed} character varieties is a disjoint union of four circles determined by the parities $\ell_\alpha$ and $\ell_\beta$.

Suppose instead that $\ep$ is non-zero and small. 
Recall from Example (b) that 
$i_1(\TT^3)\subset \chi(F_2)_0.$  The diffeomorphism $d_\partial$ sends the commutator $[r_-,s_-]$ to its inverse, and hence preserves $\kappa$.  
Therefore
the image of $i_2$ again lies in $\chi(F_2)_0$.
Since $i_{1,\ep}(t, \alpha, \beta)\in i_2(\TT^3)$,
$$0=\kappa(t, \alpha, \beta)=\cos(\tfrac\pi 2-2\ep \sin t),$$
and hence $t=0 $ or $\pi$.  One can check by hand that the eight resulting points, $(t,\alpha, \beta)=(\pm \tfrac \pi 2,0,\pm \tfrac \pi 2)$ or $(\pm \tfrac \pi 2, \pi, \pm \tfrac \pi 2)$, lie in $i_2(\TT^3)$ and are transverse intersection points, but this is better understood as a gauge-theoretic consequence of the fact that the perturbation curve represents $r_-\overline{s_+}$, and $\Real(R_- \overline{S_+})=\cos t$  is Morse with two critical points on each of the four circles  in the unperturbed intersection 
$$i_1\left(\left\{e^{\theta\bbi}\right\}\times \left\{ (\pm \tfrac\pi 2,\pm \tfrac\pi 2)\right\}\right)=\chi(H_1,W_1)\cap \chi(H_2,W_2). $$ These four circles are Morse-Bott critical levels of the Chern-Simons function of the closed 3-manifold $(S^3,L)$, and hence   (\cite[Lemma 7]{Boden-Herald}) the holonomy perturbed Chern-Simons function is Morse with eight critical points. This latter fact implies that the intersection is transverse, using Weil's theorem and the Mayer-Vietoris sequence.
 
\subsection{The 5-complex $\chi(F_2)_1$}
 In contrast to $\kappa_{\CC P^3}$, $\kappa$ is not Morse-Bott  over $[-1,1]$. In fact, the level set $\kappa^{-1}(1)=\chi(F_2)_1$ is a 5-dimensional CW-complex homotopy equivalent to $S^2\times S^2$, but is not a manifold, as we next explain.

Consider the surjective map 
 $\Psi\colon \TT^4\times[0,\pi]\to \chi(F_2)_1$ given by
  \begin{equation}\label{pss}\Psi(\alpha_-,\beta_-,\alpha_+,\beta_+,\gamma)=  \left[\left(e^{\alpha_-\bbi},~
   e^{\beta_-\bbi},~ e^{\tfrac\gamma2 \bbk}e^{\alpha_+\bbi}e^{-\tfrac\gamma2 \bbk},~ e^{\tfrac\gamma2 \bbk}e^{\beta_+\bbi}e^{-\tfrac\gamma2 \bbk}\right)\right]. \end{equation}
 The parameter $\gamma\in [0,\pi]$ determines the angle between  the maximal torus containing  the  images of $r_-,s_-$ and the maximal torus containing the images of $r_+,s_+$.

Define a $(\ZZ/2)^2$ action on $\TT^4\times I=(\RR/2\pi\ZZ)^4\times[0,\pi]$   generated by the two involutions
 $$(\alpha_-,\beta_-,\alpha_+,\beta_+,\gamma)\mapsto(-\alpha_-,-\beta_-,-\alpha_+,-\beta_+,\gamma)
 $$
  $$(\alpha_-,\beta_-,\alpha_+,\beta_+,\gamma)\mapsto (\alpha_-,\beta_-,-\alpha_+,-\beta_+,\pi-\gamma).
 $$
 Then $\Psi$ factors through the quotient, inducing  a surjection
   $$\overline{\Psi}\colon (\TT^4\times I)/\left( \ZZ /2\right)^2\to \chi(F_2)_1.$$

The following proposition summarizes the properties  of $\overline{\Psi}$;  we leave the proof  to the reader.  Define $\left( \chi(F_2)_1 \right)^*$ to be the open dense set of conjugacy classes of representations that restrict noncentrally to both $\pi_1(F_-)$ and $\pi(F_+)$.  
 \begin{proposition} \label{tori} \hfill
 \begin{enumerate}
\item  The  fiber  of  $\overline{\Psi}$ over a point $\rho$ is a singleton if $\rho\in \left( \chi(F_2)_1 \right)^*$; otherwise the  fiber  is an interval.

\item The deformation retract of $\TT^4\times [0,\pi]$ to $\TT^4\times \left\{\tfrac \pi 2\right\}$
 given by $(\tau,\gamma,s)\mapsto \left(\tau,(1-s)\gamma+ s\tfrac \pi 2\right)$ is $\left( \ZZ/2\right) ^2 $-equivariant and induces a deformation retraction of $\chi(F_2)_1$ to $\Psi(\TT^4\times \left\{\tfrac \pi 2\right\})$.

 \item \label{third part} The image of 
 $$\TT^4\times \left\{\tfrac \pi 2\right\}\xrightarrow{\Psi}\chi(F_2)_1$$
 is homeomorphic to $S^2\times S^2$, and this map has degree four onto its image.   
 
 \item \label{fourth part} The image of 
 $$\TT^4\times \{0\}\xrightarrow{\Psi}\chi(F_2)_1$$
 equals the abelian locus $\chi^\ab(F_2)\cong \TT^4/\{\pm 1\}$, and this map has degree two onto its image.

 \end{enumerate}
 \end{proposition}
Since $\chi(F_2)_1$ is homotopy equivalent to $S^2\times S^2$, it cannot be a closed manifold of dimension other than four.  But the proposition shows that it contains an open 5-ball. Hence $\kappa$ is not Morse-Bott  near $\kappa=1$, for any smooth structure on $\chi(F_2)$.

  \subsection{Representatives of $H_*(\chi(F_2))$}\label{homology}
 We now provide several explicit embedded 2- and 4-manifolds in $\chi(F_2)$  representing various homology classes.    
 Let  $x\in H_2(\chi(F_2))$  and $y\in H_4(\chi(F_2))$  be generators and denote by $z= \Psi_*([\TT^4\times \{\gamma \}])\in H_4(\chi(F_2))$ for any $\gamma\in [0,\pi]$. Parts (\ref{third part}) and (\ref{fourth part}) of Proposition \ref{tori} implies that $4[S^2\times S^2]=z$, and $2[\chi^\ab(F_2)]=z$.

\subsubsection{Representing the class $2y$ by an embedded $S^2\times S^2$}  
 Part (\ref{third part}) of  Proposition \ref{tori} implies that 
 the image $Q':=\Psi(\TT^4\times\{\tfrac \pi 2\})$ is homeomorphic to $S^2\times S^2$ and meets each fiber of $\phi\colon \chi(F_2)_1\to \chi(F_-)\times_\kappa  \chi(F_+) _1$ once, so the restriction $G\colon Q'\to Q$ is injective and hence is a homeomorphism.   Since $Q$ represents twice a generator of $H_4(\CC P^3)$, it follows that $Q'$ represents   $\pm2y\in H_4(\chi(F_2))$.

\subsubsection{Representing the class  $4y$ by $\chi^\ab(F_2)$}
Using the previous example,    $[\chi^\ab(F_2)] =2[Q']=\pm 4y$.

\subsubsection{Representing the generator  $x\in H_2(\chi(F_2))$ by an embedded 2-sphere} 
Consider the disk ${\mathbb D}$ of radius $\tfrac \pi 2$, in polar coordinates $\alpha \in [0, \tfrac \pi 2]$ and $\theta \in S^1$.  Define two maps 
$$ A_1\colon  {\mathbb D}\to \chi(F_2),~A_1(\alpha, \theta)=
\left[\left(  e^{\alpha\bbi},   e^{\alpha\bbj},  e^{\theta I(\alpha)}   e^{\alpha\bbj}e^{-\theta I(\alpha)}  ,   e^{\theta I(\alpha)} e^{\alpha\bbi}e^{-\theta I(\alpha)}\right)\right] \text{ and }$$
$$A_2\colon  {\mathbb D}\to \chi(F_2),~A_2(\alpha, \theta)= \left[\left(  e^{\alpha\bbi},   e^{\alpha\bbj},  -e^{\alpha\bbk}e^{\theta I(\alpha)}   e^{\alpha\bbj}e^{-\theta I(\alpha)}  e^{-\alpha\bbk},   -e^{\alpha\bbk}e^{\theta I(\alpha)} e^{\alpha\bbi}e^{-\theta I(\alpha)} e^{-\alpha\bbk}\right)\right],$$
where $$I(\alpha)=\frac{\sin\alpha \bbi -\sin\alpha\bbj +\cos\alpha \bbk}{\left\| \sin\alpha \bbi -\sin\alpha\bbj +\cos\alpha \bbk\right\|}.$$
Note that 
$$
[e^{\alpha \bbi},e^{\alpha\bbj}]= (\cos2\alpha+2\cos^2\alpha\sin^2\alpha)
+2 \cos\alpha\sin^2\alpha(\sin\alpha \bbi -\sin\alpha\bbj +\cos\alpha \bbk).
$$
It follows easily that $A_1$ and $A_2$ take values in $\chi(F_2)$.  

\begin{proposition} The maps $A_1,A_2$ are embeddings of the disk onto opposite hemispheres of an embedded  2-sphere in $\chi(F_2)$.  This 2-sphere  represents the generator $x\in H_2(\chi(F_2))$ and intersects the zero section of $E$ transversely in two points  with the same sign.  
\end{proposition}
\noindent{\sl Sketch of proof.} 
Since $A_1(\tfrac\pi 2, \theta)=A_2(\tfrac\pi 2, -\theta)$, the two circles at $\alpha=\tfrac \pi 2$ are identified with reversed orientation.  Moreover $A_1$ and $A_2$ disjointly embed each open disk. Thus the union of the (oriented) disks is an  oriented embedded 2-sphere in $\chi(F_2)$.  By construction, this sphere intersects $Q'$  in two points with the same sign, so  it represents $\pm x$.\qed

\subsection{Further remarks} Other examples of multicurves associated to 
holonomy perturbed    $SU(2)$ character varieties of 4-ended tangles in the 3-ball can be found 
in   \cite{FKP, HHK, HHK3,KaiS}.  K. Smith's  article \cite{KaiS} describes the  Lagrangian correspondence $${\rm Curv}(\chi(S^2,4))\times {\rm Curv}(\chi(S^2,4))\to{\rm Curv}(\chi(S^2,4))$$ induced by Conway's {\em Tangle Addition} six-stranded tangle in a twice punctured
3-ball.  This correspondence can be postcomposed with  the Lagrangian correspondence $\mathcal{L}_{(H_0,T)}$ in Corollary \ref{correspond} to produce 
a Lagrangian correspondence   taking pairs of multicurves in $\chi(S^2,4)$ to Lagrangian immersions in $\chi^*(F_2)$.  
The articles \cite{CHKK, KaiS} contain further examples which highlight the need to   keep track of  hidden information,  in the form of {\em bounding cochains}, as explained in \cite{Fukaya, Bottman-Wehrheim, CHKK}, in order to obtain a well-defined invariant of closed 3-manifolds.   
    
 \medskip    
 
The action of the mapping class group $PSL(2,\ZZ)$ of $(S^2,4)$ on $\chi(S^2,4)$
is the tautological one \cite[Theorem B]{CHK}. We leave the challenge of describing  the genus-two mapping class group action  on   ${\rm Lag}(\chi(F_2))$ to future work.

\appendix

  \section{Proof of Theorem \ref{th4.1K6}}\label{APPB}
   \begin{lemma}\label{lem4.3}
Suppose $R_-, S_-, R_+, S_+\in SU(2)$ satisfy $[R_-,S_-][R_+,S_+]=1$.  
Set 
$$
W=\left\{\Ima(R_-), \Ima \left(\overline{S_-}~\overline{R_-}\right),   \Ima(S_+ S_-),  \Ima \left(
\overline{S_- }R_+ \overline{S_+}\right),  \Ima \left(\overline{R_+ }S_- \right)\right\}\subset su(2).
$$

Then $\dim {\rm Span}(W)<3$.  Moreover, $\dim {\rm Span}(W)<2$ if and only if $R_-,S_-,R_+,$ and $S_-$ generate an abelian subgroup of $SU(2)$. 

Choose  $x_1\in S^2$ to be a unit vector in perpendicular to ${\rm Span}(W)$ and define $x_2,\dots, x_6\in SU(2)$ by:
 \begin{equation}\label{exes} x_2=\overline{x_1} R_-, ~ x_3=\overline{R_-} x_1\overline{ S_-},~ x_4=S_-\overline{x_1}S_+,~ x_5=\overline{S_+}x_1\overline{S_-}R_+,~ x_6=\overline{R_+}S_-\overline{x_1}.\end{equation}
Then
\begin{equation}
\label{require}
\Real(x_i)=0,~x_1 x_2x_3x_4x_5 x_6=1,~ R_-=x_1x_2,~S_-=\overline{x_3}~ \overline{ x_2},\text{ and }  R_+=x_4x_5,~ S_+=\overline{x_6}~\overline{x_5}.
\end{equation}

Moreover,
\begin{itemize}
\item If $v,w\in W$ is any pair satisfying  $vw\ne wv$ then $\dim {\rm Span}(W)=2$ and $x_1=\pm \frac{vw-wv}{||vw-wv||} $. No   choice other than (\ref{exes}) for $x_2,\dots, x_6$ satisfies the conclusions (\ref{require}).
\item If all pairs of elements of $W$ commute, then $R_-,S_-,R_+,$ and $S_-$ generate an abelian subgroup of $SU(2)$.   Moreover, $\dim {\rm Span}(W)=0$ or $1$, and, correspondingly, there exists an $S^2$ or $S^1$ of choices for $x_1$.  No   choice other than (\ref{exes}) for $x_2,\dots, x_6$ satisfies the conclusions (\ref{require}).
 \end{itemize}
 
\end{lemma}
Assuming Lemma \ref{lem4.3}, the proof of Theorem \ref{th4.1K6} is completed as follows. Given
$[\rho]\in \chi(F_2)$, choose a representative $\rho$, and set $R_\pm=\rho(r_\pm)$ and $S_\pm=\rho(s_\pm)$. 
If $\rho$ is non-abelian, then 
 Lemma \ref{lem4.3} implies that $\dim {\rm Span}(W)=2$, and that  the fiber of $p^*$ over $[\rho]$ consists of a pair of distinct points, corresponding to the choice of sign of $x_1$, and hence  precisely one free $\nu$ orbit. 
 
If $ \rho $ is abelian but non-central, then $\dim {\rm Span}(W)=1$, and hence there exists 
a unit vector $P\in {\rm Span}(W)$, unique up to sign.   If  $x_1\perp W$ as given in the lemma, then  $e^{\theta P}x_1e^{-\theta P}$, $\theta\in S^1$, parameterizes the other solutions for $x_1$.  It follows that the fiber of $p^*$ over $[\rho]$ is a single point, fixed by $\nu$.
Finally, if $ \rho $ is central, then $\dim {\rm Span}(W)=0$, and since conjugation acts transitively on $S^2$,  
the fiber of $p^*$ over $[\rho]$ is a single point, fixed by $\nu$.

If $\rho$ is abelian, conjugating $\rho$ if necessary, we may assume that $W\subset\RR \bbk$ and hence we may choose $x_1=\bbi$.  Therefore the fiber over $[\rho]$, seen in the previous paragraph to be a single point,  has a  binary dihedral representative. 
It follows that the fixed point set is $\chi^{\rm fix}(S^2,6)$, and 
 $p^*(\chi^{\rm fix}(S^2,6))= \chi^{\rm ab}(F_2).$ \qed
 
\subsection{Proof of Lemma \ref{lem4.3}}
 Write $$a=R_-,~b=\overline{S_-}~\overline{R_-}, c=S_+S_-, ~ d= \overline{S_-}R_+ \overline{S_+}, ~e= \overline{R_+} S_-,$$ so that $W=\{\Ima(a),\Ima(b),\Ima(c),\Ima(d),\Ima(e)\}$.  If each pair of elements of $W$ commute, then $a,b,c,d$ and $e$ are contained in a maximal torus $\{e^{\alpha P}\}$. Hence $W\subset{\rm Span}\{\Ima(P)\}$ and so $\dim {\rm Span}(W)\leq 1$.  It is routine to verify that if each  pair of elements of $W$ commute, then $R_-,S_-,R_+,$ and $S_-$ generate an abelian subgroup of $SU(2)$.

Assume that some pair in $W$ does not commute.  Then clearly $\dim {\rm Span}(W)\ge 2$. 
 Observe that
$$
abcde=R_-\overline{S_-}~\overline{R_-} S_+ S_-\overline{S_-} R_+\overline{S_+}\overline{R_+} S_-=\overline{S_-} [S_-,R_-][S_+,R_+]S_-=1
$$ and $$
edcba=\overline{R_+} S_- \overline{S_-} R_+ \overline{S_+} S_+ S_- \overline{S_- }~\overline{ R_- }R_-=1,
$$
and therefore $abcd=dcba=\overline{e}$.
Let $W'=\{\Ima(a),\Ima(b),\Ima(c),\Ima(d)\}\subset W.$
If $\dim {\rm Span}(W')<3$, there there exists an $x\in S^2$ so that $\Real(xa)=\Real(xb)=\Real(xc)=\Real(xd)=0$.  This implies that 
$$\Real(xe)=\Real(x\overline{abcd})=
\Real(x\overline{d}\overline{c} \overline{b}\overline{a})
=\Real(d x \overline{c} \overline{b}\overline{a})=
\Real(d cba x)=\Real(\overline{ e} x)=-\Real(xe),
$$
and hence $\Ima(e)\in {\rm Span}(W')$, so that ${\rm Span}(W)={\rm Span}(W')$.  Hence it suffices to prove $\dim {\rm Span}(W')<3.$

\medskip

Suppose first that $b$ and $c$ do not commute. Since $abcd=dcba$, it follows that $abcd\overline{a}=dcb$ and $\overline{d} abcd=cba$. Hence 
 $\Real(cba)=\Real(abc)\text{ and }\Real(bcd)=\Real(dcb).$   Therefore
 $$\Real(a(bc-cb))=0=\Real(d(bc-cb)).$$
Set $x_1=\tfrac{bc-cb}{\|bc-cb\|}$. Then $0=\Real(x_1a)=\Real(x_1b)=\Real(x_1c)=\Real(x_1d)$, and hence $\dim {\rm Span}(W')<3$.

Suppose instead that   $b$ and $c$ commute.  Since $abcd=dcba=dbca$, $$\overline{(bc)} (\overline{d} a) (bc)=  a\bar d .$$ If $\overline{d} a=\pm1$, then there exists $Q\in S^2$ and $\gamma\in\RR$ so that
$a= e^{\gamma Q}=\pm d$.  Then ${\rm Span}(W')\subset{\rm Span}\{P,Q\}$ and so $\dim {\rm Span}(W')\leq 2$.
On the other hand, if  $\overline{d} a\ne\pm1$, then $a\overline{d}=a(\overline{d} a)\overline{a}$, so that $[\overline{a} ~\overline{c}\overline{b}, \overline{d} a]=1$.
Write $\overline{a}~ \overline{c}\overline{b}=e^{\gamma Q}$ and $\overline{d} a=e^{\tau Q}$.  Then $a=\overline{c}\overline{b} e^{-\gamma Q}$
and $d=\overline{c}\overline{b} e^{-(\gamma+\tau) Q}$.
Hence 
$$
\Ima(a)=\Ima(\overline{c}\overline{b} e^{-\gamma Q})=\cos\gamma\Ima(\overline{c}\overline{b}) -\sin\gamma\Ima(\overline{c}\overline{b} Q)
$$ and $$
\Ima(d)=\Ima(\bar c\bar b e^{-(\gamma+\tau) Q})=\cos(\gamma+\tau)\Ima(\bar c\bar b) -\sin(\gamma+\tau)\Ima(\bar c\bar bQ).
$$
Let $P\in S^2$ so that $b,c\in \{e^{\theta P}\}$.  Clearly,
$$\Ima(a),\Ima(b),\Ima(c),\Ima(d)\in {\rm Span}\{P,\Ima(\overline{c}\overline{b}Q)\},$$
so that $\dim {\rm Span}(W')\leq 2$ also.  Hence $\dim {\rm Span}(W)<3$ for any choice of $R_-,S_-,R_+,S_+$.

\medskip

 Given {\em any}  $x_1\in SU(2)$, it is trivial to check that the assignment  (\ref{exes}) satisfies $$x_1 x_2x_3x_4x_5 x_6=1,~ R_-=x_1x_2,~S_-=\overline{x_3}~\overline{x_2}, R_+=x_4x_5,~ S_+=\overline{x_6}~\overline{x_5}.$$
If, in addition,  $x_1\in {\rm Span}(W)^\perp\cap S^2$, then $\Real(x_2)=\Real(\overline{x_1}R_-)=\Real(\overline{x_1}\Ima(R_-))=-\Real( x_1\Ima(a))=0$, and, similarly, $\Real(x_i)=0$ for $i=3,4,5,6$.

Conversely, if $x_1,x_2,\dots ,x_6\in S^2$ satisfy
$$x_1 x_2x_3x_4x_5 x_6=1,~ R_-=x_1x_2,~S_-=\overline{x_3}~ \overline{x_2}, R_+=x_4x_5,~ S_+=\overline{x_6}~\overline{x_5},$$
then $R_-, S_-, R_+, S_+$ and $x_1$ obviously determine $x_2$ and $x_3$. Then $x_6=-\overline{x_6}=x_1x_2x_3x_4x_5=
x_1x_2x_3R_+$ is determined by $R_-, S_-, R_+, S_+$ and $x_1$. Finally $x_5=\overline{S_+}~\overline{x_6}$ and hence $x_4=R_+\overline{x_5}$ are also determined by $R_-, S_-, R_+, S_+$ and $x_1$.  Hence no other choices for the $x_i$ satisfy the conclusions.

\medskip

If $v,w\in W$ is any pair satisfying   $vw\ne wv$,   then $v$ and $w$  are linearly independent and hence span ${\rm Span}(W)$; in particular $\dim {\rm Span}(W) =2$.
The vector  $x=\frac{vw-wv}{\| vw-wv \|}$  is perpendicular to $v$ and $w$.  Hence $x=\pm x_1$, and in particular is independent of $v$ and $w$ up to sign.
\qed

\end{document}